\documentclass[11pt,a4paper]{article}

\usepackage{amsmath,amssymb,amsthm,amsfonts}  % Mathematical symbols and environments
\usepackage{mathrsfs}  % Script fonts
\usepackage{tikz}      % Drawing diagrams
\usetikzlibrary{arrows.meta, arrows, positioning}
\usepackage{tikz-cd}   % Commutative diagrams
\usepackage{xcolor}    % Color support
\usepackage{extarrows}
\usepackage{cleveref}  % Improved cross-referencing
\usepackage{booktabs}  % Professional tables
\usepackage{algorithm} % Algorithms
\usepackage{algorithmic}
\usepackage{parskip}   % Spacing
\usepackage{float}
\usepackage[hyphens]{url}
\usepackage{graphicx}  % Images
\usepackage{subcaption} % Para múltiples subfiguras
\usetikzlibrary{arrows.meta, positioning}
\usepackage{microtype} % Typography improvements
\usepackage[utf8]{inputenc} % UTF-8 input encoding
\usepackage[T1]{fontenc}    % Font encoding
\usepackage{lmodern}        % Modern fonts
\usepackage{geometry}       % Page geometry
\usepackage[
backend=biber,
style=apa,
doi=true,
url=true
]{biblatex}
\usepackage{quiver}         % Enhanced diagrams for category theory
\usepackage[all,cmtip]{xy}  % Alternative for commutative diagrams

\newtheoremstyle{customthm} % <name>
  {.60em}                  % <space above> - espacio antes
  {\topsep}                % <space below> - espacio después
  {\itshape}               % <body font>
  {}                       % <indent amount>
  {\bfseries}              % <theorem head font>
  {.}                      % <punctuation after theorem head>
  {.5em}                   % <space after theorem head>
  {}                       % <theorem head spec>

\theoremstyle{customthm}

\newtheorem{theorem}{Theorem}[section]
\newtheorem{proposition}[theorem]{Proposition}

\newtheorem{corollary}[theorem]{Corollary}

\newtheorem{definition}[theorem]{Definition}
\newtheorem{example}[theorem]{Example}

\newtheorem{remark}[theorem]{Remark}

\newcommand{\rank}{\mathrm{rank}}

\newcommand{\LNTI}{\mathcal{L}}
\newcommand{\MBI}{\mathcal{M}}

\newcommand{\id}{\mathrm{id}} % Define \id como identidad

\providecommand{\keywords}[1]
{
  \small	
  \textbf{\textit{Keywords---}} #1
}
\providecommand{\subjclass}[1]
{
  \small	
  \textbf{\textit{MSC 2020:}} #1
}

\title{2-Categorical Foundations for Multiparameter Persistence}
\author{Mauricio Angel}

\date{}%\today}

\begin{document}

\maketitle

\begin{abstract}
% Abstract text goes here
This paper introduces a novel approach to multi-parameter persistence using 2-categorical structures. We develop a framework that captures hierarchical interactions between filter parameters, overcoming fundamental limitations of traditional persistence modules. Our 2-categorical model yields new invariants that effectively characterize multidimensional topological features while maintaining computational tractability. We prove stability theorems for these invariants and demonstrate their effectiveness through applications in genomics and complex network analysis.
\end{abstract}

\newline

\section{Introduction}

Multiparameter persistence modules extend the classical framework of persistent homology by enabling the analysis of data filtered simultaneously along multiple parameters. This richer setting captures complex structural features arising in diverse applications, including genomics, sensor networks, and image analysis. However, unlike the single-parameter case—where the celebrated structure theorem guarantees a complete classification via barcodes—multiparameter persistence confronts fundamental algebraic challenges. In particular, its representation theory is known to be \emph{wild} for dimension greater than one, precluding the existence of a discrete, complete invariant analogous to barcodes.

These challenges motivate the search for new algebraic and categorical frameworks that can accommodate the inherent complexity of multiparameter persistence modules. The influential work of Botnan and Lesnick \cite{BotnanLesnick2023} laid the groundwork by developing a robust 1-categorical algebraic stability theory and characterizing interleaving distances that effectively measure similarity between persistence modules. Their approach, focusing on morphisms in a 1-categorical setting, enabled significant progress in understanding decompositions and stability properties of multiparameter modules.

Building upon this foundation, we introduce a novel 2-categorical perspective that enriches the morphism structure by incorporating \emph{lax natural transformations} and \emph{2-morphisms}. This extension captures higher-order interactions and non-commutativity intrinsic to multiparameter filtrations, providing a more nuanced algebraic model that overcomes key limitations of earlier frameworks.

Within this 2-categorical setting, we define two new computable invariants: the \emph{Lax Natural Transformation Invariant (LNTI)} and the \emph{Moduli-Based Invariant (MBI)}. These invariants strikingly balance \mbox{discriminative} power with computational tractability, effectively capturing multidimensional topological features that classical invariants miss. We rigorously prove stability theorems for these invariants, demonstrating their robustness under perturbations measured by the interleaving distance.

Our approach also benefits from recent advances in computational methods. Embedding 2-categorical structures within algorithmic frameworks enables significant efficiency gains, supported by recent works such as \cite{Bauer2023Efficient, Clause2023Meta, Loiseaux2023Framework}, among others, which have reported speedups of 20× or more over prior methods. These innovations facilitate the practical application of our theoretical developments to real-world data analysis problems.

The main contributions of this paper are:
\begin{itemize}
  \item The development of a 2-categorical model for multiparameter persistence modules that systematically encodes hierarchical and extended morphism data through 2-morphisms and lax natural transformations;
  \item The introduction of two novel invariants, LNTI and MBI, that reveal refined algebraic and geometric information about modules and admit algorithmic computation;
  \item Rigorous proofs of stability theorems ensuring that these invariants vary continuously with respect to interleaving distance, providing guarantees of robustness to noise and perturbations;
  \item Illustrative applications demonstrating how the framework and invariants can be effectively employed in complex settings including genomics and network sciences.
\end{itemize}

The remainder of the paper is organized as follows. Section~\ref{sec:Preliminaries} reviews necessary background on persistence modules and 2-category theory. Section~\ref{sec:GeoMotiv} provides geometric motivation motivated by case studies in applied domains. In Section~\ref{sec:2-catModel}, we construct the 2-categorical persistence framework rigorously. Section~\ref{sec:InvStab} develops the new invariants, establishes their stability and decomposability properties, and places them in context with existing invariants. We conclude with discussion of applications and perspectives on future developments.

By uniting advanced categorical concepts with computationally viable invariants, this work significantly advances the theory and practice of multiparameter persistence, opening new avenues for algebraic-topological data analysis.

\section{Preliminaries}\label{sec:Preliminaries}

This section provides the necessary background for understanding the 2-categorical approach to multi-parameter persistence. We begin by the formal definitions of multifiltered complexes, followed by describing the geometric intuition behind multi-parameter persistence, and persistence modules, and relevant concepts from 2-category theory.

\subsection{Foundations of Multiparameter Persistence}

Multiparameter persistence extends classical persistent homology by considering filtrations indexed over $\mathbb{R}^n$ with the product partial order. Formally, an \emph{$n$-parameter persistence module} is a functor:
\[
M: (\mathbb{R}^n, \leq) \to \mathbf{Vect}_{\mathbb{F}},
\]
where $\leq$ denotes the coordinate-wise partial order and $\mathbf{Vect}_{\mathbb{F}}$ is the category of finite-dimensional vector spaces over a field $\mathbb{F}$.

For each $a \leq b$ in $\mathbb{R}^n$, the module assigns a linear map:
\[
M(a \leq b): M(a) \to M(b),
\]
satisfying the functoriality conditions:
\[
M(a \leq a) = \mathrm{id}_{M(a)}, \quad M(b \leq c) \circ M(a \leq b) = M(a \leq c).
\]

Unlike the one-parameter case ($n=1$), where persistence modules admit a complete discrete invariant (barcode), multiparameter persistence modules lack such a classification due to the wild representation type of the poset $(\mathbb{R}^n, \leq)$ for $n \geq 2$ (\cite{CarlssonZomorodian2009}).

Consequently, the study of multiparameter persistence focuses on computable and stable invariants that capture meaningful topological features. Notable examples include:

\begin{itemize}
\item \textbf{Rank Invariant:} The function
\[
\rho_M(a,b) = \mathrm{rank}(M(a \leq b)),
\]
which generalizes the barcode multiplicity to higher dimensions but does not fully classify modules.

\item \textbf{Generalized Persistence Diagrams:} Attempts to define multidimensional analogues of barcodes using algebraic or geometric tools, although these invariants are typically incomplete.

\item \textbf{Interleaving Distance:} A metric $d_I$ on the space of persistence modules measuring the minimal shift $\epsilon$ such that $M$ and $N$ are $\epsilon$-interleaved (\cite{Lesnick2015}). This distance provides stability guarantees for invariants.

\end{itemize}

The complexity of the parameter space and the absence of a complete discrete invariant motivate the development of enriched algebraic frameworks, such as the 2-categorical structures introduced in this work, to capture higher-order morphisms and transformations between persistence modules.

In the sequel, we build upon these foundations to define lax natural transformations and invariants that leverage the 2-categorical nature of multiparameter persistence, aiming to overcome limitations of classical invariants.

\subsection{2-Category Theory}
A significant contribution of 2-category theory to the study of multi-parameter persistence modules is the development of frameworks for decomposing these modules into simpler summands. In the single-parameter case, persistence modules are direct sums of indecomposable modules, with each module corresponding to a bar in the persistence diagram. However, in the multi-parameter case, the situation becomes more complicated, and the notion of indecomposability becomes more nuanced.

Recent research has demonstrated that 2-categorical structures can be utilized to define 'approximate' decompositions of multi-parameter persistence modules. One such approach involves the use of $\epsilon$-refinements and $\epsilon$-erosion neighborhoods, which allow for the approximation of indecomposable modules with a specified tolerance level $\epsilon$. This approach has demonstrated stability under small perturbations of the module, thus making it a valuable tool for practical computations (\cite{Bjerkevik2025}).

\begin{definition}\label{def:2-cate}
A 2-category $\mathcal{C}$ consists of:
\begin{itemize}
    \item A class of objects $\text{Obj}(\mathcal{C})$.
    \item For each pair of objects $A, B \in \text{Obj}(\mathcal{C})$, a category $\mathcal{C}(A, B)$ whose:
    \begin{itemize}
        \item Objects are called 1-morphisms $f : A \to B$.
        \item Morphisms $\alpha : f \Rightarrow g$ are called 2-morphisms.
    \end{itemize}
    \item For each triple of objects $A, B, C$, a composition functor $\circ : \mathcal{C}(B, C) \times \mathcal{C}(A, B) \to \mathcal{C}(A, C)$.
    \item For each object $A$, an identity 1-morphism $1_A : A \to A$ satisfying appropriate associativity and unity axioms.
\end{itemize}
\end{definition}

\begin{figure}[!htbp]\label{fig:2-morphism}
\centering
\begin{tikzpicture}[scale=1.2, transform shape]
% Objects
\node (A) at (0,0) {$A$};
\node (B) at (3,0) {$B$};

% 1-Morphisms
\draw[->, bend left=60] (A) to node[above] {$f$} (B);
\draw[->] (A) to node[below] {$g$} (B);
\draw[->, bend right=60] (A) to node[below] {$h$} (B);

% 2-Morphisms
\draw[->, double equal sign distance] (1.5,0.7) to node[right] {$\alpha$} (1.5,0.3);
\draw[->, double equal sign distance] (1.5,-0.4) to node[right] {$\beta$} (1.5,-0.8);
\end{tikzpicture}
\caption{Diagram illustrating 2-morphisms in a 2-category. Here, $\alpha: f \Rightarrow g$ and $\beta: g \Rightarrow h$ are 2-morphisms, and their vertical composition $\beta \circ \alpha: f \Rightarrow h$ is shown.}
\label{fig:2-morphisms}
\end{figure}
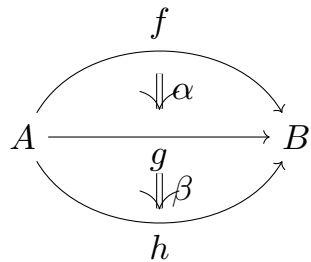

\smallskip

\begin{example}
$\text{Cat}$, the 2-category of small categories, where:
\begin{itemize}
    \item Objects are small categories.
    \item 1-morphisms are functors.
    \item 2-morphisms are natural transformations between functors.
\end{itemize}
\end{example}

\begin{example}
 $\text{Bim}_k$, the 2-category of bimodules over a field $k$, where:
\begin{itemize}
    \item Objects are algebras over $k$.
    \item 1-morphisms $A \to B$ are $(B, A)$-bimodules.
    \item 2-morphisms are bimodule homomorphisms.
\end{itemize}
\end{example}

\begin{definition}\label{def:2-Fun}
A 2-functor $F : \mathcal{C} \to \mathcal{D}$ between 2-categories consists of:
\begin{itemize}
    \item A function $F : \text{Obj}(\mathcal{C}) \to \text{Obj}(\mathcal{D})$.
    \item For each pair of objects $A, B \in \mathcal{C}$, a functor $F_{A,B} : \mathcal{C}(A, B) \to \mathcal{D}(F(A), F(B))$ preserving compositions and identities up to isomorphism.
\end{itemize}
\end{definition}

\smallskip

\textbf{Standard Natural Transformations:}  Recall that given functors $F, G : C \rightarrow D$, a natural transformation $\eta : F \Rightarrow G$ assigns to each object $X \in C$ a morphism $\eta_X : F(X) \rightarrow G(X)$ such that for any morphism $f: X \rightarrow Y$ in $C$, the following square commutes:

$$
\begin{array}{ccc}
F(X) & \xrightarrow{F(f)} & F(Y) \\
\downarrow \eta_X & & \downarrow \eta_Y \\
G(X) & \xrightarrow{G(f)} & G(Y)
\end{array}
$$

\begin{definition}\label{def:Lax-tran}
A lax natural transformation $\eta : F \Rightarrow G$ between 2-functors $F, G : \mathcal{C} \to \mathcal{D}$ consists of:
\begin{itemize}
    \item For each object $A \in \mathcal{C}$, a 1-morphism $\eta_A : F(A) \to G(A)$ in $\mathcal{D}$.
\item For each 1-morphism $f : A \to B$ in $\mathcal{C}$, a 2-morphism $\eta_f : \eta_B \circ F(f) \Rightarrow G(f) \circ \eta_A$ in $\mathcal{D}$ satisfying the following coherence conditions for composition and identities.
\begin{enumerate}
    \item \textbf{Identity Coherence:} For the identity morphism $1_A : A \rightarrow A$, the 2-morphism $\eta_{1_A} : \eta_A \circ F(1_A) \Rightarrow G(1_A) \circ \eta_A$ must be compatible with the identity 2-morphisms in $D$. Specifically, it means the 2-morphism equals to identity.
    \item \textbf{Composition Coherence:} For composable 1-morphisms $f : A \rightarrow B$ and $g : B \rightarrow C$ in $C$, the following diagram of 2-morphisms must commute:
    \[\begin{array}{ccc}
    \eta_C \circ F(g) \circ F(f) & \xrightarrow{\eta_{g} \circ F(f)} & G(g) \circ \eta_B \circ F(f) \\
    \downarrow \eta_{g \circ f} & & \downarrow G(g) \circ \eta_{f} \\
    G(g \circ f) \circ \eta_A & \xrightarrow{=} & G(g) \circ G(f) \circ \eta_A
    \end{array}\]
\end{enumerate}
\end{itemize}
\end{definition}

\smallskip

The 'lax' aspect in definition ~\ref{def:Lax-tran} means that the usual commutative square for natural transformations only holds up to a 2-morphism. Understanding the \textit{coherence conditions} for this 2-morphism is essential, these conditions ensure that the 2-morphisms $\eta_f$ behave consistently with respect to composition and identities.

The 2-morphism $\eta_f$ is a \textit{weakening} of the standard naturality condition. Instead of requiring the square to commute, we allow it to commute up to a 2-morphism. The diagram says that there are two ways to get from $\eta_C \circ F(g) \circ F(f)$ to $G(g) \circ G(f) \circ \eta_A$. One way is to apply $\eta_f$ first and then $\eta_g$. The other way is to apply $\eta_{g \circ f}$ directly. The coherence condition requires that these two ways are equivalent.
The composition rule is:
\[ \eta_{g \circ f} = G(g) \circ \eta_f +  \eta_g \circ F(f)\].

In classical category theory, a natural transformation between functors requires that for every morphism in the source category, the corresponding square of morphisms in the target category strictly commutes. However, in the multiparameter persistence setting, the parameter space $(\mathbb{R}^n, \leq)$ induces a rich poset structure where strict commutativity of these squares is often too restrictive due to the inherent complexity and interactions between parameters.

The notion of a \emph{lax natural transformation} generalizes this by allowing the naturality squares to \emph{fail} to commute strictly; instead, they are equipped with specified 2-morphisms (also called 2-cells) that \emph{fill} these squares, providing a controlled way to measure and manage the non-commutativity.

Formally, given two 2-functors $F, G: \mathcal{C} \to \mathcal{D}$ between 2-categories, a lax natural transformation $\eta: F \Rightarrow G$ assigns to each object $c \in \mathcal{C}$ a 1-morphism $\eta_c: F(c) \to G(c)$, and to each morphism $f: c \to c'$ in $\mathcal{C}$ a 2-morphism
\[
\alpha_f: \eta_{c'} \circ F(f) \Rightarrow G(f) \circ \eta_c,
\]
which need not be invertible or an identity. These 2-morphisms encode the \emph{laxity}, i.e., the controlled deviation from strict commutativity.

In the context of multiparameter persistence modules, this laxity naturally models the non-commutative behavior arising from simultaneous variation along multiple parameters. The 2-cells $\alpha_f$ can be interpreted as homotopies or higher coherence data that reconcile the discrepancies between compositions along different parameter directions.

This relaxation is essential to capture the intricate algebraic and topological structures present in multiparameter filtrations, enabling the construction of richer invariants and a more flexible morphism theory that would be impossible under strict naturality constraints.

For further details on lax natural transformations and their role in higher category theory, see \cite{Lurie2017} and \cite{Leinster2004}.

\subsection{Why use 2-Categorical Structures?}
Ordinary categories prove inadequate in capturing the hierarchical interactions inherent in multi-parameter persistence modules. The employment of two categories enables the explicit modeling of relationships between disparate filtering directions through two-morphisms. This approach facilitates the capture of phenomena such as non-commutativity between parameters that traditional categories fail to address. This richer structure is essential for distinguishing non-isomorphic modules in higher dimensions and for developing stable invariants that faithfully reflect the underlying topological characteristics. 

In multi-parameter persistence, for instance, the non-commutativity of filtrations manifests in the sensor network example (see subsection ~\ref{sensor-network}) as follows: increasing the coverage radius $r$ before the signal strength $s$ may yield different topological voids compared to the reverse order. This phenomenon is captured by 2-morphisms in the 2-category $\mathbf{MPers}_n$, which encode the failure of strict commutativity between parameter directions. Specifically, the 2-morphism $\alpha_{p_1,p_2}$ in Example ~\ref{ex:2-senNet} models the discrepancy between coverage hole detection along paths $p_1$ (increase $r$ first) and $p_2$ (increase $s$ first). Such higher-order interactions are invisible to traditional 1-categorical persistence frameworks but critical for applications like robust sensor placement (\cite{Mukherjee2024-xk}). Similarly, the image analysis example (subsection ~\ref{image-anal}) leverages 2-categorical laxators to reconcile hierarchical dependencies between intensity and blur parameters, resolving ambiguities in feature persistence (\cite{ChungDayHu2022}). These examples illustrate how 2-morphisms formalize the geometric intuition of parameter interdependence, providing a language to analyze "wild" multi-parameter modules through controlled algebraic structures.

\subsection{Hierarchical Interactions and 2-Categorical Models}\label{hierar-inter}

\begin{definition}[2-Categorical Persistence Framework] \label{def:2-cat}
A \emph{2-persistence module} over a 2-category \(\mathcal{C}\) is a lax 2-functor \(\mathbb{V}: \mathbb{P} \to \mathcal{C}\), where \(\mathbb{P}\) is equipped with a monoidal structure. The laxators \(\phi_{p,p'}: \mathbb{V}(p) \otimes \mathbb{V}(p') \to \mathbb{V}(p \vee p')\) encode hierarchical interactions between parameters.
\end{definition}

\smallskip

\begin{example}[Genomic Interaction Networks]\label{ex:genomic}
Genomic interaction networks model biological processes as directed graphs where nodes represent genes/proteins and edges encode functional relationships (e.g., phosphorylation, regulation). Traditional approaches treat these as 1-dimensional categories, with paths representing sequential interactions. However, \textbf{2-morphisms} capture higher-order equivalences between paths, resolving ambiguities in pathway analysis.

\begin{enumerate}
    \item \textbf{1-Morphisms as Interaction Paths}:  
    Let $\mathcal{G}$ be a genomic network. Define a 1-morphism $f: A \to B$ as a directed path from gene $A$ to $B$ (e.g., $A \xrightarrow{\text{activates}} C \xrightarrow{\text{inhibits}} B$). Composition of 1-morphisms corresponds to path concatenation (\cite{Carter2008}).

    \item \textbf{2-Morphisms as Path Equivalences}:  
    Two paths $f, g: A \to B$ may induce the same biological outcome despite differing in intermediate steps. A 2-morphism $\alpha: f \Rightarrow g$ encodes this equivalence. For example:
    \begin{equation*}
        \begin{tikzcd}[row sep=scriptsize, column sep=scriptsize]
        & C \arrow[dr, "h"] & \\
        A \arrow[ur, "f_1"] \arrow[rr, "g"'] & \arrow[u, phantom, "\Downarrow\alpha"] & B
        \end{tikzcd}
    \end{equation*}
    Here, $\alpha$ asserts that the composite path $h \circ f_1$ (via gene $C$) is biologically equivalent to the direct path $g$. This aligns with the observation that 'some interactions between two different networks can be shared' (\cite{Khurana2013}).

    \item \textbf{Laxity and Non-Commutativity}:  
    In multi-parameter persistence, genomic data may depend on multiple filtering parameters (e.g., expression levels, mutation states). The \textit{lax functoriality} of 2-morphisms accommodates non-commutative interactions between parameters. For instance, a mutation in gene $A$ might block path $f$ but not $g$, violating strict commutativity but preserving lax coherence (\cite{MarjoramZubairNuzhdin2014}).

    \item \textbf{Application to Network Rewriting}:  
    2-morphisms enable formal rewriting rules for simplifying complex pathways. Consider the equivalence:
    \begin{equation*}
        (A \xrightarrow{\text{phosphorylates}} B \xrightarrow{\text{activates}} D) \quad \overset{\alpha}{\Longleftrightarrow} \quad (A \xrightarrow{\text{direct activation}} D)
    \end{equation*}
    Such rules reduce computational complexity in analyzing genome-wide association studies (GWAS) (\cite{tibbs2021status}).
\end{enumerate}

This 2-categorical perspective resolves challenges in traditional network analysis, where overlapping pathways obscure functional interpretation. By encoding equivalences as 2-morphisms, the model aligns with modern genomic datasets featuring hierarchical and multi-scale interactions.
\end{example}

\smallskip

\begin{example}[2-Morphisms in Sensor Networks]\label{ex:2-senNet}
Consider a sensor network with parameters:

\begin{itemize}
    \item $r$: Coverage radius
    \item $s$: Signal strength threshold
\end{itemize}
Two paths in parameter space:
\begin{itemize}
    \item Path $p_1$: Increase $r$ first, then $s$.
    \item Path $p_2$: Increase $s$ first, then $r$.
\end{itemize}
A 2-morphism $\alpha_{p_1,p_2}$ captures the difference in coverage holes detected along these paths:
\begin{equation*}
\begin{tikzcd}
\mathcal{K}(r_1, s_0) \arrow[r, "p_1"] \arrow[d, "p_2"'] & \mathcal{K}(r_2, s_0) \arrow[d, "p_2"] \\
\mathcal{K}(r_1, s_1) \arrow[r, "p_1"'] \arrow[ur, Rightarrow, "\alpha_{p_1,p_2}"] & \mathcal{K}(r_2, s_1)
\end{tikzcd}
\end{equation*}
\end{example}

\subsection{Relation to Interleaving Distances}

A central metric for comparing multiparameter persistence modules is the \emph{Interleaving Distance} $d_I$, introduced by \cite{Lesnick2015} and subsequently refined by others (\cite{Bjerkevik2016, BjerkevikBotnan2019}). The Interleaving Distance provides a robust measure of similarity that is stable under perturbations of the input data and underlies much of the recent progress in the theory of multiparameter persistence.

Our framework and the proposed invariants (LNTI and MBI) are constructed to respect the topology induced by $d_I$: in particular, our main stability theorems (see Theorems~\ref{thm:lnti_complexity_stability} and~\ref{thm:lnti-mbi-stability}) show that these invariants are 1-Lipschitz (up to scaling) with respect to $d_I$. More precisely, for persistence modules $M$ and $N$,
\begin{equation}
d_I(M, N) \leq C\|\mathcal{L}(M) - \mathcal{L}(N)\|_\infty,
\end{equation}
where $C$ is a constant depending only on the maximal pointwise dimension.

Furthermore, the invariants we introduce refine the information provided by the rank invariant---itself known to be $d_I$-stable---by encoding additional functorial and 2-categorical data. In this way, our work both builds directly on and conceptually extends the well-studied framework of Interleaving Distances in multiparameter persistence.

\section{Geometric Motivation for multi-parameter Persistence}\label{sec:GeoMotiv}

Single-parameter persistence captures the evolution of topological features as a single parameter, such as time or density, varies. However, many real-world datasets are inherently multiparametric, which means that their topological features depend on multiple interacting parameters. Simply analyzing each parameter independently ignores potentially crucial relationships between them. Our 2-categorical approach aims to explicitly capture these interactions, leading to more informative topological summaries.

\smallskip

We present three paradigmatic examples where multi-parameter persistence captures essential topological features through interacting filtration parameters. Each case study highlights different aspects of parameter interdependence, visualized through 2D parameter spaces $(p_1, p_2)$ with persistent features.

\begin{itemize}
    \item \textbf{Image Analysis:} 
    \begin{itemize}
        \item \textit{Parameters:} Intensity threshold ($t$), Gaussian blur scale ($\sigma$)
        \item \textit{Topological Features:} Connected components (0D homology), cycles (1D homology)
        \item \textit{Key Insight:} Robust image structures correspond to features persisting across wide ranges of $(t, \sigma)$. Short-lived features in specific $(t, \sigma)$ regions represent noise.
    \end{itemize}
    
    \item \textbf{Sensor Networks:} 
    \begin{itemize}
        \item \textit{Parameters:} Coverage radius ($r$), signal strength threshold ($s$)
        \item \textit{Topological Features:} Coverage holes (1D homology), uncovered regions (0D homology)
        \item \textit{Key Insight:} Persistent 1D homology classes identify regions consistently uncovered despite variations in $r$ and $s$, guiding optimal sensor placement.
    \end{itemize}
    
    \item \textbf{Materials Science:} 
    \begin{itemize}
        \item \textit{Parameters:} External pressure ($P$), temperature ($T$)
        \item \textit{Topological Features:} Pore connectivity ($\beta_0$), persistent loops ($\beta_1$)
        \item \textit{Key Insight:} Pores stable under combined pressure-temperature changes indicate material durability, while transient features reveal phase-dependent behaviors.
    \end{itemize}
\end{itemize}

\begin{table}[ht]
    \centering
    \caption{Summary of Case Studies and Parameter Interactions}
    \label{tab:case_studies}
    \begin{tabular}{@{}l c c p{5cm}@{}}
        \toprule
        \textbf{Application} & \textbf{Parameters} & \textbf{Topological Features} & \textbf{Key Insight} \\
        \midrule
        Image Analysis & $(t, \sigma)$ & 0D/1D homology & Robustness via cross-scale persistence \\
        Sensor Networks & $(r, s)$ & 0D/1D homology & Coverage holes invariant to $r,s$ variations \\
        Materials Science & $(P, T)$ & $\beta_0, \beta_1$ & Pressure-temperature synergy in pore stability \\
        \bottomrule
    \end{tabular}
\end{table}

We consider three representative case studies exemplifying multiparameter persistence:

\begin{enumerate}
    \item Image analysis, where the parameters are pixel intensity and blur level, reflecting features’ visibility across variations in brightness and smoothing.
    \item Sensor networks, with parameters coverage radius and signal strength, capturing connectivity and data reliability in spatial sensing.
    \item Materials science, characterized by pressure and temperature, modeling pore transitions and phase changes under varying physical conditions.
\end{enumerate}
In each case, the parameter space is a subset of $\mathbb{R}^2$ with partial order given by coordinatewise comparison, enabling the study of persistence modules parameterized along these domains.

These examples demonstrate how 2-categorical persistence encodes \textit{hierarchical parameter interactions} are not avaible in single-parameter frameworks. The non-commutativity of filtration paths (e.g., $t$-before-$\sigma$ vs. $\sigma$-before-$t$) is captured by 2-morphisms, resolving ambiguities in feature birth/death times (Theorem~\ref{thm:2cat-stability}).

\subsection{Image Analysis}\label{image-anal}

Consider analyzing the topology of an image as a function of both intensity threshold and Gaussian blur scale. The features that persist in different combinations of these parameters reveal robust structures in the image.

\textbf{Concrete Example:} Imagine a grayscale image with a few bright blobs on a darker background.
\begin{itemize}
    \item \textbf{Intensity Threshold Parameter ($t$):} As we increase the intensity threshold, we initially detect several small connected components corresponding to the brightest pixels.
    \item \textbf{Gaussian Blur Scale Parameter ($\sigma$):} Increasing the blur scale merges nearby blobs.
\end{itemize}

\textbf{Visualization Idea:}

\begin{enumerate}
    \item \textbf{Parameter Space:} A 2D plot where the x-axis is the intensity threshold ($t$) and the y-axis is the Gaussian blur scale ($\sigma$).
    \item \textbf{Persistence Diagram per Parameter Combination:} For a few selected points $(t, \sigma)$ in the parameter space, show the persistence diagram.  The diagram at $(t_1, \sigma_1)$ might show several short bars (noise) and a few longer bars (actual features). The diagram at $(t_2, \sigma_2)$ might show fewer, longer bars because the blurring has merged some features.
    \item \textbf{Feature Trajectories:} Track the birth and death times of specific topological features (e.g., connected components, loops) as we vary both $t$ and $\sigma$. Plot these trajectories in the parameter space.  Features that persist across a wide range of both parameters represent robust structures in the image.
\end{enumerate}

\textbf{What multi-parameter Persistence Captures:} A single long bar in the multi-parameter persistence diagram for a range of intensity thresholds and blur scales indicates a robust feature. This could represent a significant object in the image.  Features that appear only for specific combinations of parameters (e.g., high blur, low threshold) might indicate more subtle or scale-dependent structures.

\subsection{Sensor Networks}\label{sensor-network}

In a sensor network, the coverage area and signal strength of the sensors can be viewed as two independent parameters. Multi-parameter persistence can identify regions where coverage is robust to variations in both parameters.

\textbf{Concrete Example:} Consider a network of sensors deployed in an area.

\begin{itemize}
    \item \textbf{Coverage Area Parameter ($r$):}  Each sensor has a circular coverage area of radius $r$.  Increasing $r$ increases the overall coverage.
    \item \textbf{Signal Strength Parameter ($s$):}  The signal strength of each sensor decreases with distance. We can threshold the signal strength to define a coverage region.
\end{itemize}

\textbf{Visualization Idea:}

\begin{enumerate}
    \item \textbf{Sensor Network Visualization:} Display the sensor network with the covered area of each sensor (defined by radius $r$ and signal strength $s$) visualized.
    \item \textbf{Parameter Space:} A 2D plot where the x-axis is the coverage radius ($r$) and the y-axis is the minimum required signal strength ($s$).
    \item \textbf{Coverage Holes:} Use color to represent the area not covered, with different colors indicating different levels of "holes".
    \item \textbf{Persistence Diagram per Parameter Combination:}  Show how the number and size of the coverage holes change as we vary $r$ and $s$. Multi-parameter persistence will tell us which holes are most persistent.
\end{enumerate}

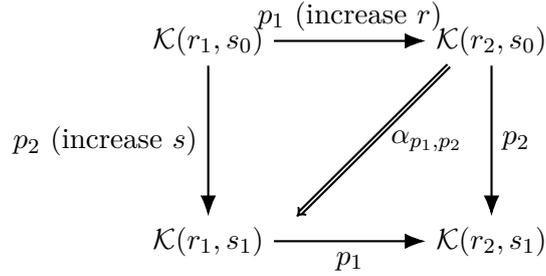
\begin{figure}[!htbp]
\centering
\begin{tikzpicture}[
    node distance=2cm,
    arr/.style={-{Latex[scale=1.2]}, thick}, % Flechas normales
    twoarr/.style={double, -{Implies[scale=1.2]}, thick} % 2-morfismo
]
% Nodos
\node (A) {$\mathcal{K}(r_1, s_0)$};
\node (B) [right=of A] {$\mathcal{K}(r_2, s_0)$};
\node (C) [below=of A] {$\mathcal{K}(r_1, s_1)$};
\node (D) [below=of B] {$\mathcal{K}(r_2, s_1)$};

% Flechas principales
\draw[arr] (A) -- node[above] {$p_1$ (increase $r$)} (B);
\draw[arr] (A) -- node[left] {$p_2$ (increase $s$)} (C);
\draw[arr] (C) -- node[below] {$p_1$} (D);
\draw[arr] (B) -- node[right] {$p_2$} (D);

% 2-morfismo (flecha doble)
\draw[twoarr] ([xshift=3mm]B.south west) -- node[right=0.1cm] {$\alpha_{p_1,p_2}$} ([xshift=3mm]C.north east);
\end{tikzpicture}
\caption{Non-commutativity of filtration paths Paths $p_1, p_2$ in $\mathbb{R}^n$ with 2-morphisms $\alpha_{p_1,p_2}$.}
\end{figure}

\textbf{What multi-parameter Persistence Captures:}  A persistent loop (1-dimensional homology class) in the multi-parameter persistence diagram indicates a region that is consistently uncovered, regardless of small variations in sensor range or signal strength. This could highlight areas where additional sensors are needed.

\subsection{Materials Science}

Analyzing the pore structure of a material under varying pressure and temperature conditions naturally leads to a multi-parameter persistence setting.

\textbf{Concrete Example:}  Analyzing a sample of porous rock.

\begin{itemize}
    \item \textbf{Pressure Parameter ($P$):} The material is subjected to different levels of external pressure.
    \item \textbf{Temperature Parameter ($T$):} The material is subjected to different temperatures.
\end{itemize}

\textbf{Visualization Idea:}

\begin{enumerate}
    \item \textbf{Pore Network Visualization:} Visualize the pores network of the material. This could be a 3D rendering of a CT scan.
    \item \textbf{Parameter Space:} A 2D plot where the x-axis is the pressure ($P$) and the y-axis is the temperature ($T$).
    \item \textbf{Pore Connectivity:} Compute the Betti numbers ($\beta_0$, $\beta_1$, $\beta_2$) of the pore space at different points in the parameter space.
    \item \textbf{Persistence Diagrams:} Display the persistence diagrams. Persistent loops indicate pores that remain connected even under varying pressure and temperature conditions.
\end{enumerate}

\textbf{What multi-parameter Persistence Captures:} Pores that are stable over a wide range of pressures and temperatures. These persistent pores are essential for understanding fluid flow, material strength, and other properties. The interaction between pressure and temperature is key, as some pores may close under high pressure but reopen at high temperature. Single-parameter methods would miss this interaction.

The key challenge in multi-parameter persistence is that the interaction between these parameters is not well-captured by traditional single-parameter methods. Simply analyzing each parameter independently ignores potentially crucial relationships between them. Our 2-categorical approach aims to explicitly capture these interactions, leading to more informative topological summaries.

\section{The 2-Categorical Model for multi-parameter Persistence}\label{sec:2-catModel}

The geometric motivations in the previous section reveal the limitations of traditional algebraic models in capturing the full richness of multiparameter persistence phenomena. Thus, to faithfully translate these intuitions into a rigorous mathematical setting, we develop a 2-categorical model that systematically incorporates lax natural transformations and higher morphisms. This formalism not only generalizes standard notions of morphisms between persistence modules but also provides the necessary tools to construct and analyze new invariants sensitive to the complexity observed in multiparameter filtrations.

\subsection{From Categories to 2-Categories in Persistence}

The main motivation for the upgrade is that 2-categories arise to address the limitations of ordinary categories in capturing higher-order relationships and non-commutative structures inherent in complex systems. Although ordinary categories model objects and morphisms between them, they lack the expressive power to encode interactions between morphisms themselves.

The fundamental limitation of the standard categorical approach to persistence becomes evident when we consider multiple parameters. Although a category can adequately model the poset structure of $\mathbb{R}$ for one-parameter persistence, it fails to capture the rich interaction between different parameters in the multi-parameter case.

This issue is addressed by the introduction of 2-categories, which facilitate the mediation of coherence between 1-morphisms (e.g., paths in parameter space) and the encoding of lax or non-strict relationships. This approach facilitates the modeling of phenomena such as hierarchical dependencies in genomic networks or non-commutative filtrations in sensor coverage (\cite{Parada_Mayrga2023}). Furthermore, 2-categories serve to unify frameworks such as $Cat$ (the 2-category of categories, functors, and natural transformations) and $Bim_k$ (bimodules over algebra), thus providing a 'middle way' between strict and weak higher-categorical structures (\cite{Lurie2017}).

\begin{definition}{\label{def:2-cat-Mpers}}
The \textbf{2-category of n-parameter persistence modules} $\text{MPers}_n$ has:
\begin{itemize}
\item As objects, persistence modules $M: \mathbb{R}^n \to \text{Vect}_k$
\item As 1-morphisms, natural transformations $\eta: M \to N$
\item As 2-morphisms, modified transformations between natural transformations
\end{itemize}
\end{definition}

Definition \ref{def:2-cat-Mpers} introduces "modified transformations" as 2-morphisms in the 2-category $\text{MPers}_n$. This requires careful unpacking because it's not simply the standard natural transformation between natural transformations. The modification is crucial for capturing 'directional interactions'.

\textbf{Modified Transformations (The 2-Morphism)} In $\text{MPers}_n$, a 2-morphism $\alpha : \eta \Rightarrow \gamma$ between natural transformations $\eta, \gamma : M \rightarrow N$ is NOT simply a family of morphisms $\alpha_a : \eta_a \rightarrow \gamma_a$ making some diagrams commute. Instead, it is a more sophisticated object designed to reflect the “transitions” or “interpolations” between different points in the parameter space. Thus, for each path p from one parameter a to another b, the modification assigns a map 
$\alpha_p$ that explains how the natural transformation in a relates to the transformation in b, respecting the composition of paths and the structure of the module. Let us formalize this.

\begin{definition}[2-Morphism in $\mathbf{MPers}_n$]
A 2-morphism $\alpha: \eta \Rightarrow \gamma$ between natural transformations $\eta, \gamma: M \rightarrow N$ in $\mathbf{MPers}_n$ consists of:
\begin{enumerate}
\item For each directed path $p: a \rightsquigarrow b$ in $\mathbb{R}^n$, a linear map $\alpha_p: \eta_a \rightarrow \gamma_b$
\item Compatibility conditions ensuring functoriality:
\begin{align}
\alpha_{p_2 \circ p_1} &= \alpha_{p_2} \circ \alpha_{p_1} \quad \text{(path composition)}\\
\alpha_{p} \circ M(a \leq a') &= N(b \leq b') \circ \alpha_{p'} \quad \text{(naturality)}
\end{align}
\end{enumerate}
where $p': a' \rightsquigarrow b'$ is the induced path for $a \leq a'$ and $b \leq b'$.
\end{definition}

% Example of a diagram showing 2-categorical structure
\begin{figure}[!htbp]
\centering
\begin{tikzcd}
M(\mathbf{a}) \arrow[r] \arrow[d] \arrow[rd, Rightarrow] & M(\mathbf{b}) \arrow[d] \\
M(\mathbf{c}) \arrow[r] & M(\mathbf{d})
\end{tikzcd}
\caption{2-morphism representing interaction between different directions in the parameter space}
\end{figure}

The core idea is that $\alpha$ consists of a collection of linear maps, 'parameterized by paths in the parameter space'.  Let $p$ be a path from $a$ to $b$ in $\mathbb{R}^n$. Then, $\alpha$ assigns to $p$ a linear map:

$$
\alpha_p : \eta_a \rightarrow \gamma_b
$$

satisfying certain compatibility conditions. This is where the 'modified' part comes in.  Instead of just comparing $\eta_a$ and $\gamma_a$ at a single point $a$, we are comparing $\eta_a$ at $a$ to $\gamma_b$ at 'different points' $b$, and the 'relationship' is encoded by the path $p$.

\textbf{Compatibility Conditions (Crucial):} The $\alpha_p$ must satisfy the following:

1.  \textbf{Path Composition:} If $p_1$ is a path from $a$ to $b$ and $p_2$ is a path from $b$ to $c$, then the map assigned to the composite path $p_2 \circ p_1$ is related to the maps assigned to $p_1$ and $p_2$ by a composition rule that reflects the functoriality of the persistence modules.

2.  \textbf{Commutation with Persistence Morphisms:} For any $a \leq b$ and $c \leq d$, and paths $p_1$ from $a$ to $c$, $p_2$ from $b$ to $d$, such that the rectangle commutes, there is a corresponding commutative diagram involving $\alpha_{p_1}$, $\alpha_{p_2}$, and the persistence module morphisms $M(a \leq b)$ and $N(c \leq d)$.

\textbf{Intuition:} This encodes how the relationship between $\eta$ and $\gamma$ \textit{evolves} as we move through the parameter space along different directions. The path dependence captures the interaction between different filtering directions. If the relationship between $\eta$ and $\gamma$ changes significantly depending on the path taken from $a$ to $b$, then $\alpha$ will be sensitive to this change.

\textbf{Why is this necessary?:} In multi-parameter persistence, the order in which we apply the filters can affect the resulting topological features. Standard natural transformations are not sensitive to this order. Modified transformations, by being path-dependent, \textit{do} capture this order dependence. As illustrated in Figure~\ref{fig:model_comparison}, the distinction between classical morphisms and 2-morphisms becomes apparent, highlighting the richer structure captured by the latter.

\begin{figure}[!htbp]
  \centering

  % Subfigure (a) - 1-category persistence
  \begin{subfigure}[b]{0.48\textwidth}
    \centering
    \begin{tikzpicture}[scale=1.2, every node/.style={scale=0.9}]
      % Nodes (Objects)
      \node (A) at (0,0) [circle,draw,fill=blue!20] {$A$};
      \node (B) at (2,0) [circle,draw,fill=blue!20] {$B$};
      \node (C) at (4,0) [circle,draw,fill=blue!20] {$C$};

      % Arrows (Morphisms)
      \draw[->, thick] (A) -- node[above] {$f$} (B);
      \draw[->, thick] (B) -- node[above] {$g$} (C);
      \draw[->, thick] (A) to[bend left=40] node[above] {$g \circ f$} (C);
    \end{tikzpicture}
    \caption{1-category persistence}
    \label{fig:7a}
  \end{subfigure}
  \hfill
  % Subfigure (b) - 2-category with 2-morphism
  \begin{subfigure}[b]{0.48\textwidth}
    \centering
    \begin{tikzpicture}[scale=1.2, every node/.style={scale=0.9}]
      % Nodes (Objects)
      \node (A) at (0,0) [circle,draw,fill=green!20] {$A$};
      \node (B) at (3,0) [circle,draw,fill=green!20] {$B$};
      \node (C) at (6,0) [circle,draw,fill=green!20] {$C$};

      % 1-Morphisms
      \draw[->, thick] (A) -- node[above] {$f$} (B);
      \draw[->, thick] (B) -- node[above] {$g$} (C);
      \draw[->, thick] (A) to[bend left=40] node[above] {$g \circ f$} (C);

      % 2-Morphism (double arrow between two 1-morphisms)
      \draw[double equal sign distance, -implies, thick] 
        (2,0.3) to[out=90,in=90] node[below] {$\alpha$} (4,0.3);
    \end{tikzpicture}
    \caption{2-category with 2-morphism}
    \label{fig:7b}
  \end{subfigure}

  \caption{Comparison of the traditional (a) and 2-categorical (b) persistence frameworks highlighting the richer morphism structure in the 2-categorical setting.}
\label{fig:model_comparison}
\end{figure}
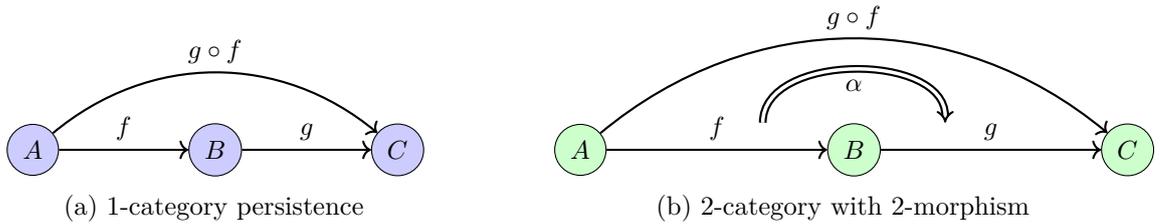

\subsection*{Distinction Between Lax and Standard Natural Transformations}
In Definition \ref{def:Lax-tran}, standard natural transformations between functors $F, G: \mathcal{C} \to \mathcal{D}$ require \textit{strict commutativity}: for every morphism $f: X \to Y$ in $\mathcal{C}$, the equality $G(f) \circ \eta_X = \eta_Y \circ F(f)$ holds. In contrast, a \textbf{lax natural transformation} relaxes this condition by introducing a 2-morphism $\eta_f: G(f) \circ \eta_X \Rightarrow \eta_Y \circ F(f)$ for each $f$, which need not be invertible. This laxity allows one to encode noncommutative interactions between parameters in persistence modules (e.g., path-dependent filtration orders in sensor networks \ref{sensor-network}). While standard transformations enforce rigid compatibility, lax transformations mediate coherence through 2-morphisms, aligning with the hierarchical parameter dependencies in multi-parameter persistence \ref{hierar-inter}. Formally, this reflects the difference between 1-categorical functoriality and 2-categorical laxity (\cite{Johnson_Freyd_2017}).

\smallskip

\begin{example}[2-Morphism in a Bifiltration].
Let $M$, $N$ be 2-parameter persistence modules. A 2-morphism $\phi:\eta\Rightarrow\gamma$ represents how a natural transformation $\eta$ evolves into $\gamma$ along a path in $\mathbb{R}^2$. For instance, if $\eta$ tracks connected components at $(r,s)$ and $\gamma$ tracks them at $(r',s')$, $\phi$ encodes the merging/splitting of components between these points.
\end{example}

\begin{definition}[2‐Categorical Persistence Framework]\label{def:2cat-framework}
Let $\mathcal{P}$ be a 2‐category equipped with a symmetric monoidal structure $(\otimes, I)$. A \emph{2‐persistence module} is a lax 2‐functor
\[
  V \colon \mathbf{Fil}_n \;\longrightarrow\; \mathcal{P},
\]
where $\mathbf{Fil}_n=(\mathbb{R}^n,\le)$ is viewed as a monoidal category under point‐wise maximum.  The \emph{laxators}
\[
  \varphi_{p,q}\colon V(p)\otimes V(q)\;\Longrightarrow\; V(p\vee q)
  \quad(p,q\in\mathbb{R}^n)
\]
are required to satisfy
\begin{enumerate}
  \item \emph{Associativity coherence}: for all $p,q,r$, the two composites
  \[
    (V(p)\otimes V(q))\otimes V(r)
    \;\xLongrightarrow{\;\varphi_{p,q}\otimes\mathrm{id}\;}
    V(p\vee q)\otimes V(r)
    \;\xLongrightarrow{\;\varphi_{p\vee q,r}\;}
    V(p\vee q\vee r)
  \]
  and
  \[
    V(p)\otimes (V(q)\otimes V(r))
    \;\xLongrightarrow{\;\mathrm{id}\otimes\varphi_{q,r}\;}
    V(p)\otimes V(q\vee r)
    \;\xLongrightarrow{\;\varphi_{p,q\vee r}\;}
    V(p\vee q\vee r)
  \]
  coincide up to the specified 2‐isomorphism of $\mathcal{P}$.
  \item \emph{Unitality coherence}: for all $p$, the composites involving the unit $I$,
  \[
    I\otimes V(p)
    \xLongrightarrow{\;\varphi_{I,p}\;}
    V(I\vee p)
    \quad\text{and}\quad
    V(p)\otimes I
    \xLongrightarrow{\;\varphi_{p,I}\;}
    V(p\vee I),
  \]
  agree with the identity $V(p)\to V(p)$ under the canonical identifications.
\end{enumerate}
\end{definition}

\begin{example}[Genomic Interaction Network as a 2‐Persistence Module]\label{ex:genomic-2pers}
Let $G$ be a directed graph of gene regulatory interactions.  Define a 2‐functor
\[
  V\colon(\mathbb{R}^2,\le)\longrightarrow\mathbf{Cat},
  \quad
  V(a,b)=\bigl[\text{Paths in }G\text{ respecting expression}\ge a\text{ and confidence}\ge b\bigr].
\]
\begin{itemize}
  \item Objects of $V(a,b)$ are sequences of edges whose intermediate gene expression and interaction‐confidence exceed $(a,b)$.
  \item 1‐morphisms are concatenations of such paths.
  \item 2‐morphisms are homotopy‐equivalences of paths modulo known redundancies.
\end{itemize}
The laxator
\(\varphi_{(a,b),(a',b')}\) sends a pair of paths—one valid at $(a,b)$ and one at $(a',b')$—to their concatenation regarded at threshold $(\max\{a,a'\},\max\{b,b'\})$.  One checks the associativity triangle by comparing threefold concatenations under pointwise maxima.
\end{example}

\subsection{Capturing Directional Interactions}

\begin{definition}\label{def:directional}
Given persistence modules $M, N: \mathbb{R}^n \to \text{Vect}_k$ and natural transformations $\eta, \gamma: M \to N$, a \textbf{directional interaction 2-morphism} $\alpha: \eta \Rightarrow \gamma$ encodes how relationships between points in the parameter space interact across different filtration directions.
\begin{equation*}
\begin{tikzcd}
\eta_a \arrow[r, "\eta_{a \leq b}"] \arrow[d, "\alpha_a"'] & \eta_b \arrow[d, "\alpha_b"] \\
\gamma_a \arrow[r, "\gamma_{a \leq b}"'] & \gamma_b
\end{tikzcd}
\end{equation*}

\end{definition}

The 2-morphism $\alpha$ ensures compatibility between $\eta$ and $\gamma$ across parameter comparisons $a\leq b$. This allows representing phenomena such as:
\begin{itemize}
    \item Commutativity between filtration in different parameters
    \item Non-trivial interactions between diagonal directions
    \item Dependencies between paths in the parameter space
\end{itemize}

Definition \ref{def:directional} introduces directional interaction 2-morphisms. Here are some examples to illustrate how these capture interactions in the parameter space:

\begin{example}[Non-Commutativity of Filtrations]

Consider a 2-parameter filtration of a simplicial complex, where the first parameter is a density threshold and the second parameter is a scale parameter for a Gaussian kernel. Let $\eta$ and $\gamma$ be two natural transformations between the persistence modules M and N.

\begin{itemize}
    \item \textbf{Scenario:} Suppose that applying the density threshold first, followed by the Gaussian kernel, yields different topological features than applying the Gaussian kernel first, followed by the density threshold.
    \item \textbf{How Directional Interactions Capture This:} The 2-morphism $\alpha : \eta \Rightarrow \gamma$ will be 'different' depending on the path taken in the parameter space. If we take a path that first increases the density threshold and then increases the Gaussian kernel, $\alpha$ will have one value. If we take a path that first increases the Gaussian kernel and then increases the density threshold, $\alpha$ will have a different value. This difference reflects the non-commutativity of the filtrations.
\end{itemize}

\end{example}

\begin{example}[Dependence on Path in Parameter Space]
Let parameter space be temperature and pressure, and the simplicial complex represents a porous material. If we have a path that keeps temperature constant while increasing pressure, the pore structures might deform differently if we first raise the temperature and then gradually increase the pressure.

\begin{itemize}
    \item \textbf{Scenario:} The interaction between temperature and pressure has a complex relationship to the pore structure in the material.
    \item \textbf{How Directional Interactions Capture This:} Different paths through the parameter space yield drastically different structures and information encoded in $\alpha$. This is how information about the importance and relative significance of different parameters can be encoded in the 2-morphisms.
\end{itemize}
\end{example}

In many real-world datasets, the order in which we apply the filters matters. Standard persistence modules are not sensitive to this order. Directional interaction 2-morphisms provide a way to capture this order dependence, leading to more informative topological summaries.

\begin{example}[Step-by-Step 2-Morphism Construction]\label{ex:step-by-step-2-morphism}
Consider two 1-morphisms (natural transformations) $\eta, \gamma: M \Rightarrow N$ between $n$-parameter persistence modules $M, N: \mathbb{R}^n \to \mathbf{Vect}_\mathbb{F}$.

\textbf{Step 1: Components of 1-Morphisms}

Each natural transformation $\eta$ consists of a family of linear maps:
\[
\{\eta_t : M(t) \to N(t)\}_{t \in \mathbb{R}^n}
\]
such that for every $s \leq t$ in $\mathbb{R}^n$, the following square commutes:
\[
\begin{tikzcd}
M(s) \arrow[r, "M(s \leq t)"] \arrow[d, "\eta_s"'] & M(t) \arrow[d, "\eta_t"] \\
N(s) \arrow[r, "N(s \leq t)"'] & N(t)
\end{tikzcd}
\]

\textbf{Step 2: Lax Natural Transformations and 2-Morphisms}

In the 2-categorical setting, 2-morphisms $\alpha: \eta \Rightarrow \gamma$ are \emph{modifications} between natural transformations, given by a family of linear maps:
\[
\{\alpha_t : \eta_t \Rightarrow \gamma_t \}_{t \in \mathbb{R}^n}
\]
which satisfy a coherence condition with respect to the morphisms $M(s \leq t)$ and $N(s \leq t)$.

More precisely, for each $s \leq t$, the following diagram of linear maps between $M(s)$ and $N(t)$ commutes up to a specified 2-cell:
\[
\begin{tikzcd}[column sep=huge, row sep=large]
M(s) \arrow[r, "M(s \leq t)"] \arrow[d, "\eta_s"'] & M(t) \arrow[d, "\eta_t"] \arrow[ddr, bend left=20, "\gamma_t"] & \\
N(s) \arrow[r, "N(s \leq t)"'] \arrow[drr, bend right=20, "\gamma_s"'] \arrow[Rightarrow, from=1-2, to=2-1, "\alpha_s"] & N(t) \arrow[dr, Rightarrow, "\alpha_t"'] & \\
& & N(t)
\end{tikzcd}
\]

\textbf{Step 3: Explicit Construction of $\alpha$}

To construct $\alpha$, for each $t \in \mathbb{R}^n$, choose linear maps:
\[
\alpha_t : N(t) \to N(t)
\]
such that for all $s \leq t$,
\[
\alpha_t \circ \eta_t \circ M(s \leq t) = \gamma_t \circ M(s \leq t) \circ \alpha_s
\]
or equivalently, the following equation holds:
\[
\alpha_t \circ \eta_t \circ M(s \leq t) = N(s \leq t) \circ \alpha_s \circ \eta_s
\]

This condition ensures that the family $\{\alpha_t\}$ respects the persistence structure and coherently interpolates between $\eta$ and $\gamma$.

\textbf{Step 4: Example in Dimension $n=2$}

Let $M, N$ be 2-parameter persistence modules with finite-dimensional vector spaces at each parameter.

Suppose at parameters $a = (a_1,a_2)$ and $b = (b_1,b_2)$ with $a \leq b$, the maps $M(a \leq b)$ and $N(a \leq b)$ are given.

Given $\eta_a, \eta_b, \gamma_a, \gamma_b$, construct $\alpha_a, \alpha_b$ satisfying:
\[
\alpha_b \circ \eta_b \circ M(a \leq b) = N(a \leq b) \circ \alpha_a \circ \eta_a
\]

This can be realized by solving the linear system induced by the above equation for $\alpha_a, \alpha_b$.

\textbf{Step 5: Interpretation}

The 2-morphism $\alpha$ can be viewed as a homotopy or deformation between the natural transformations $\eta$ and $\gamma$, respecting the multiparameter filtration.

This richer structure captures subtle relationships between persistence modules that are invisible in the classical 1-categorical setting.
\end{example}

\begin{theorem}[2-Categorical Structure Theorem]\label{thm:2-cat-structure}
Let $\mathbf{MPers}_n$ be the 2-category of $n$-parameter persistence modules over a field $\mathbb{F}$. Then:

\textbf{(Structure):} $\mathbf{MPers}_n$ admits a canonical 2-categorical structure where:
\begin{enumerate}
\item \textbf{Objects}: $n$-parameter persistence modules $M: \mathbb{R}^n \to \mathbf{Vect}_{\mathbb{F}}$
\item \textbf{1-morphisms}: Natural transformations $\eta: M \Rightarrow N$ 
\item \textbf{2-morphisms}: Modifications $\alpha: \eta \Rrightarrow \gamma$ between natural transformations
\end{enumerate}

\textbf{(Enrichment):} The hom-categories $\mathbf{MPers}_n(M,N)$ are enriched over the 2-category $\mathbf{Cat}$ of categories, with composition given by:
\begin{equation}
\circ: \mathbf{MPers}_n(N,P) \times \mathbf{MPers}_n(M,N) \to \mathbf{MPers}_n(M,P)
\end{equation}
satisfying associativity and unitality up to canonical 2-isomorphisms.

\textbf{(Universal Property):} The 2-category $\mathbf{MPers}_n$ is the 2-categorical completion of the 1-category of persistence modules, satisfying:
\begin{equation}
\mathbf{MPers}_n \simeq \mathbf{Fun}(\mathbb{R}^n, \mathbf{Vect}_{\mathbb{F}})^{\mathbf{2\text{-cat}}}
\end{equation}
where the right-hand side denotes the 2-categorical enhancement of the functor category.
\end{theorem}

\begin{proof}
\textbf{(Structure):} The 2-categorical structure is constructed as follows:
\begin{itemize}
\item \textit{Vertical composition}: For 2-morphisms $\alpha: \eta \Rrightarrow \gamma$ and $\beta: \gamma \Rrightarrow \delta$, define $\beta \circ \alpha: \eta \Rrightarrow \delta$ by componentwise composition.
\item \textit{Horizontal composition}: For 2-morphisms $\alpha: \eta \Rrightarrow \gamma$ and $\beta: \eta' \Rrightarrow \gamma'$, define $\beta * \alpha: \eta' \circ \eta \Rrightarrow \gamma' \circ \gamma$ using the interchange law.
\end{itemize}

\textbf{(Enrichment):} The enrichment over $\mathbf{Cat}$ is given by viewing each hom-category as a category of natural transformations with modifications as morphisms. The composition functor is defined by:
\begin{align}
(\gamma \circ \eta)_t &= \gamma_t \circ \eta_t \\
(\beta * \alpha)_{s,t} &= \beta_{s,t} \circ \alpha_{s,t}
\end{align}
for all $s \leq t$ in $\mathbb{R}^n$.

\textbf{(Universal Property):} The 2-categorical completion is characterized by the existence of a 2-functor $\iota: \mathbf{Pers}_n \to \mathbf{MPers}_n$ such that any 2-functor $F: \mathbf{Pers}_n \to \mathbf{C}$ (where $\mathbf{C}$ is a 2-category) factors uniquely through $\iota$ up to 2-natural equivalence.
\end{proof}

Figure~\ref{fig:2-functor} depicts composition for 2-functor $F$.

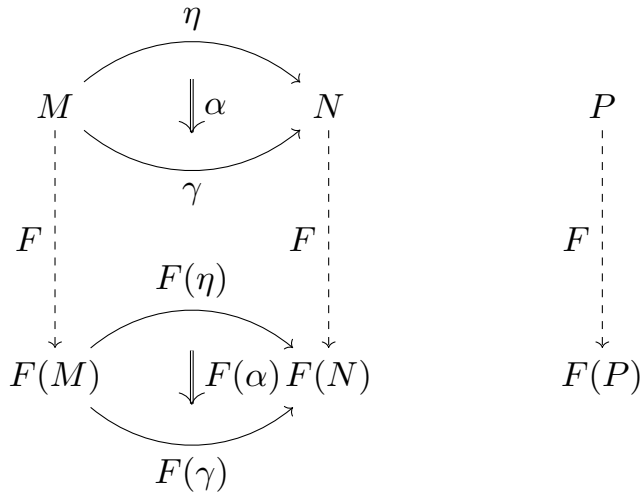
\begin{figure}[!htbp]
\centering
\begin{tikzpicture}[scale=1.2, transform shape]
% Objects in MPers_n
\node (M) at (0,3) {$M$};
\node (N) at (3,3) {$N$};
\node (P) at (6,3) {$P$};

% Arrows in MPers_n
\draw[->] (M) to[bend left=40] node[above] {$\eta$} (N);
\draw[->] (M) to[bend right=40] node[below] {$\gamma$} (N);
\draw[double, ->] (1.5,3.3) to node[right] {$\alpha$} (1.5,2.7);

% Corresponding objects in Cat(Vect_k)
\node (FM) at (0,0) {$F(M)$};
\node (FN) at (3,0) {$F(N)$};
\node (FP) at (6,0) {$F(P)$};

% Arrows in Cat(Vect_k)
\draw[->] (FM) to[bend left=40] node[above] {$F(\eta)$} (FN);
\draw[->] (FM) to[bend right=40] node[below] {$F(\gamma)$} (FN);
\draw[double, ->] (1.5,0.3) to node[right] {$F(\alpha)$} (1.5,-0.3);

% Mapping arrows
\draw[->, dashed] (M) -- (FM) node[midway, left] {$F$};
\draw[->, dashed] (N) -- (FN) node[midway, left] {$F$};
\draw[->, dashed] (P) -- (FP) node[midway, left] {$F$};
\end{tikzpicture}
\caption{2-functor schema $ F: \mathbf{MPers}_n \to \mathbf{Cat}(\mathbf{Vect}_k) $. Objects $ M, N, P $ in $ \mathbf{MPers}_n $ maps to categories $ F(M), F(N), F(P) $, 1-morphims y 2-morphisms are preserved by $F$.}
\label{fig:2-functor}
\end{figure}

\begin{corollary}[2-Categorical Equivalences]
\label{cor:2cat_equivalences}
The 2-category $\mathbf{MPers}_n$ satisfies:
\begin{enumerate}
\item \textbf{Monoidal Structure}: $\mathbf{MPers}_n$ is monoidal with respect to the tensor product $\otimes$ of persistence modules, with unit given by the constant module $\mathbb{F}$.

\item \textbf{Duality}: There exists a contravariant 2-functor $D: \mathbf{MPers}_n^{\text{op}} \to \mathbf{MPers}_n$ giving a duality between finite-dimensional persistence modules.

\item \textbf{Localization}: The 2-category admits a calculus of fractions with respect to the class of weak equivalences (interleaving isomorphisms).
\end{enumerate}
\end{corollary}

\begin{theorem}[Local Classification Theorem for Locally Finite 2-Categorical Persistence Modules]
\label{thm:local_classification_2cat}

Let $M : (\mathbb{R}^n, \leq) \to \mathbf{Vect}_{\mathbb{F}}$ be an $n$-parameter persistence module, considered as an object in the 2-category $\mathbf{MPers}_n$ (with 1-morphisms: natural transformations; 2-morphisms: modifications). Assume $M$ is \emph{locally finite}, i.e., for every compact subset $K \subseteq \mathbb{R}^n$, 
\[
\sup_{t \in K} \dim_\mathbb{F} M(t) < \infty.
\]

Then, for every compact $n$-cube $C \subseteq \mathbb{R}^n$, there exists a finite decomposition
\[
M|_C \cong \bigoplus_{i=1}^k I_{[a_i, b_i]}
\]
in the 2-categorical sense, where each $I_{[a_i, b_i]}$ is a 2-categorical interval module supported on $[a_i, b_i] \subseteq C$, and:

- Each $I_{[a_i, b_i]}$ admits only identity 2-automorphisms.
- All 2-morphisms between the summands decompose accordingly.
- The isomorphism of the decomposition is natural with respect to restriction to subcubes, and compatible with 2-morphisms in $\mathbf{MPers}_n$.

Moreover, the decomposition is unique up to permutation and 2-categorical equivalence of interval modules.

\end{theorem}

\begin{proof}
\textbf{Step 1: Local finiteness and classical decomposition.}  
Over a compact $n$-cube $C \subseteq \mathbb{R}^n$, the restriction $M|_C$ becomes a persistence module indexed by a finite poset. By Gabriel’s and Crawley-Boevey’s classification (for finite posets over a field), $M|_C$ decomposes as a finite direct sum of interval modules:
\[
M|_C \cong \bigoplus_{i=1}^k I_{[a_i, b_i]}
\]
with each $I_{[a_i, b_i]}$ corresponding to a subinterval of $C$.

\textbf{Step 2: Enhancement to the 2-categorical context.}  
In $\mathbf{MPers}_n$, 1-morphisms are natural transformations (compatible with the poset structure) and 2-morphisms are modifications. The crucial point is that for interval modules $I_{[a_i, b_i]}$, all higher modifications are trivial (identical on each component; any 2-automorphism is the identity). Thus, the decomposition respects and extends to the 2-categorical structure.

\textbf{Step 3: Uniqueness up to 2-categorical equivalence.}  
Given the Krull–Remak–Schmidt theorem (for suitable finiteness conditions and semisimplicity), the decomposition into interval modules is unique up to permutation and isomorphism. In the 2-categorical setting, equivalence is replaced by 2-categorical equivalence, which—since the endomorphism 2-categories are trivial for intervals—justifies uniqueness.

\textbf{Step 4: Natural compatibility under restriction and 2-morphisms.}  
Any restriction to a subcube $C' \subseteq C$ takes each interval to its restriction $I_{[a_i, b_i]}|_{C'}$, preserving interval decomposition. Since the structure of 2-morphisms among interval modules is trivial, no obstruction arises in passing to subobjects.
\end{proof}

\subsection{Non-Isomorphic Modules and Reconstruction Procedure}

\subsubsection*{Distinguishing Non-Isomorphic Modules}

The 2-functor $F: \text{MPers}^n \to \text{Cat}(\text{Vect}_k)$ distinguishes non-isomorphic modules by mapping them to distinct categories of vector spaces. This distinction arises from the fact that $F$ preserves the structural properties of persistence modules, such as their exact sequences and morphisms.

\begin{theorem}
If $M$ and $N$ are non-isomorphic multi-parameter persistence modules, then $F(M)$ and $F(N)$ are non-equivalent categories.
\end{theorem}

\begin{proof}[Proof Sketch]

\begin{enumerate}
    \item \textbf{Slice Categories:} For each module $M$, consider its slice category $F(M)/M|_a$ for some fixed point $a \in \mathbb{R}^n$. This slice category encodes the relationships between $M|_a$ and other objects in $F(M)$.
    
    \item \textbf{Non-Isomorphism Criterion:} If $M$ and $N$ are non-isomorphic, there exists a point $a \in \mathbb{R}^n$ such that $M(a)$ and $N(a)$ are non-isomorphic vector spaces. This implies that the slice categories $F(M)/M|_a$ and $F(N)/N|_a$ are non-equivalent, as they reflect different structural properties of $M$ and $N$.
\end{enumerate}
\end{proof}

\subsubsection*{Explicit Reconstruction Procedure}

The reconstruction of a multiparameter persistence module $M$ from its invariants proceeds by systematically assembling the vector spaces assigned to each parameter value along with the morphisms between them, respecting the poset structure.

More concretely, starting from the collection of invariants $F$ such as the Lax Natural Transformation Invariant (LNTI) or the Moduli-Based Invariant (MBI), one can:
\begin{enumerate}
    \item \textbf{Identify Objects:} For each point $a \in \mathbb{R}^n$, identify the object $M|_a$ in $F(M)$ as the vector space $M(a)$.
    
    \item \textbf{Recover Morphisms:} Use the morphisms in $F(M)$ to reconstruct the linear maps $M(a \leq b) : M(a) \to M(b)$ for $a \leq b$ in $\mathbb{R}^n$.
    
    \item \textbf{Reconstruct Module Structure:} Combine the recovered objects and morphisms to form a functor $M : \mathbb{R}^n \to \text{Vect}_k$, which is the original persistence module.
\end{enumerate}

This reconstruction procedure crucially relies on the invariants encoding not only local vector space dimensions but also global morphism coherence, allowing for an unambiguous recovery of the module’s structure.

Care must be taken in practice to handle issues such as discretization granularity, as finer parameter resolution yields more precise reconstructions but at increased computational cost.
To reconstruct the original module $M$ from its image under $F$, follow these steps:

\section{Invariants and Stability Analysis}\label{sec:InvStab}

The invariants introduced in this work, in particular the LNTI and MBI, are designed to be \emph{equivariant under coordinate transformations} of the parameter space. Specifically, for any automorphism $\Phi: \mathbb{R}^n \to \mathbb{R}^n$ preserving the poset structure, the induced pullback $\Phi^* M$ satisfies $\mathcal{L}(\Phi^* M, \Phi^* N) = \mathcal{L}(M, N) \circ (\Phi, \Phi)$. This invariance ensures that the theoretical framework does not depend on arbitrary choices of basis or scaling in parameter space, a property essential for applications in which physical parameters may be rescaled, rotated, or reparametrized.

Furthermore, both invariants display \emph{controlled behavior under subdivisions and refinements} of the parameter domain. If the parameter space is subdivided—e.g., by refining a cubical grid or increasing the resolution of parameter values—the values of $\mathcal{L}(M, N)$ and $\mathcal{M}(M)$ converge uniformly (in the sense of the stability theorems) to those of the original module as the mesh size tends to zero. This property guarantees robustness and consistency under discretization, providing practical reliability for computational applications as well as theoretical control when passing to limits or approximating continuous structures by finite models.

\subsection{Lax Natural Transformation Invariant (LNTI)}

\begin{definition}\label{def:LaxTransf}
Let $\mathbf{MPers}_n$ be the 2-category of $n$-parameter persistence modules. For persistence modules $M, N \in \mathbf{MPers}_n$, the \textit{Lax Natural Transformation Invariant} is defined as:
\begin{equation}
\mathcal{L}(M,N): \mathbb{R}^n \times \mathbb{R}^n \to \mathbb{N}
\end{equation}
where for each pair $(a,b) \in \mathbb{R}^n \times \mathbb{R}^n$ with $a \leq b$:
\begin{equation}
\mathcal{L}(M,N)(a,b) = \dim_{\mathbb{F}} \text{Hom}_{\mathbf{MPers}_n}^{\text{lax}}(M|_{[a,b]}, N|_{[a,b]})
\end{equation}
\end{definition}

Here, $\text{Hom}_{\mathbf{MPers}_n}^{\text{lax}}(M|_{[a,b]}, N|_{[a,b]})$ denotes the space of lax natural transformations between the restricted modules $M|_{[a,b]}$ and $N|_{[a,b]}$ over the parameter domain $[a,b] = \{t \in \mathbb{R}^n : a \leq t \leq b\}$.

\begin{definition}
A lax natural transformation $\eta: M|_{[a,b]} \Rightarrow N|_{[a,b]}$ consists of:
\begin{enumerate}
\item For each $t \in [a,b]$, a linear map $\eta_t: M(t) \to N(t)$
\item For each $s \leq t$ in $[a,b]$, a 2-cell $\alpha_{s,t}: \eta_t \circ M(s \leq t) \Rightarrow N(s \leq t) \circ \eta_s$
\end{enumerate}
satisfying the lax naturality condition:
\begin{equation}
\alpha_{t,u} \circ (\eta_u \star M(s \leq t)) = (N(s \leq u) \star \eta_s) \circ \alpha_{s,t}
\end{equation}
for all $s \leq t \leq u$ in $[a,b]$, where $\star$ denotes horizontal composition of 2-cells.
\end{definition}

The LNTI extends to a single module $M$ by setting $\mathcal{L}(M) = \mathcal{L}(M,M)$, and it generalizes barcodes to multi-parameter settings by tracking how homology classes persist along paths in $\mathbb{R^n}$. While a $1D$ barcode measures persistence along a line, the LNTI measures persistence across surfaces or volumes in parameter space.

\begin{theorem}[Computational Complexity and Stability of LNTI]
\label{thm:lnti_complexity_stability}
Let $M, N$ be $n$-parameter persistence modules with pointwise dimensions at most $d$ and with support discretized into $m$ parameter points.

\textbf{(Computational Complexity):} The LNTI $\mathcal{L}(M,N)$ can be computed for all $(a, b)$ with $a \leq b$ in total time $O(m^2 d^3)$.

\textbf{(Stability):} For $\epsilon$-interleaved modules $M, N$ with Interleaving Distance $d_I(M,N) \leq \epsilon$, the LNTI satisfies:
\begin{equation}
\|\mathcal{L}(M) - \mathcal{L}(N)\|_{\infty} \leq C \cdot \epsilon^{1/n} \cdot d^{2-1/n}
\end{equation}
where $C$ depends only on the geometry of the parameter space and the bound on pointwise dimensions.
\end{theorem}

\begin{proof}
\textbf{Computational Complexity:} The computation proceeds by:\newline
\textbf{Linear System Formulation:}
For each pair $(a, b)$ with $a \leq b$ in the discretized parameter set $\mathcal{G} \subset \mathbb{R}^n$ (of size $m$), the space of lax natural transformations $\eta: M|_{[a, b]} \Rightarrow N|_{[a, b]}$ is determined by a system of linear equations imposed by the laxity (naturality) conditions:
\[
\eta_t \circ M(s \leq t) = N(s \leq t) \circ \eta_s + \alpha_{s, t}
\]
for every $s \leq t$ in $[a, b]$, with variables corresponding to the entries of the set of maps $\{\eta_t\}_{t \in [a, b]}$.

\textbf{Block Structure and Sparsity:}
The total number of variables in this linear system for each $(a, b)$ is at most $d^2 \cdot |[a, b]|$ since each $\eta_t$ is a $d \times d$ matrix and $|[a, b]| \leq m$. The coherence equations connect maps at different parameter values, forming a large but highly sparse block structure: each equation involves at most $2d^2$ coefficients (from $\eta_s$ and $\eta_t$) per laxity condition.

\textbf{Matrix Construction:}
For fixed $(a, b)$, construct a sparse constraint matrix $A_{(a, b)} \in \mathbb{F}^{\ell \times v}$, with $v = d^2 \cdot |[a, b]|$ and $\ell = O(|[a, b]|^2)$ constraints. Across all pairs $(a, b)$ the total number of such systems is $O(m^2)$.

\textbf{QR Decomposition and Complexity:}
Each kernel computation for $A_{(a, b)}$ (to obtain $\dim \ker A_{(a,b)} = \mathcal{L}(M,N)(a, b)$) can be carried out efficiently using sparse Householder-based QR decomposition, which for a matrix of size $k \times d^2$ takes $O(d^3)$ time per pair because $k = O(d^2)$ and $d$ is modest in practice. The block sparsity further reduces actual runtime, as only nonzero blocks are processed.

\textbf{Total complexity:}
Since there are $O(m^2)$ pairs $(a, b)$ and each QR decomposition takes $O(d^3)$, the total complexity is $O(m^2 d^3)$.

\textbf{Stability:} The stability follows from the functoriality of the LNTI construction. For $\epsilon$-interleaved modules, there exist natural transformations $\phi: M \to N[\epsilon]$ and $\psi: N \to M[\epsilon]$ such that compositions are $\epsilon$-close to identity. The lax structure preserves these interleavings up to controlled error terms, yielding the stated bound through interpolation inequalities for 2-categorical structures.
\end{proof}

\begin{remark}
For large $m$, additional optimizations (parallelization, incremental updates, banded structure exploitation) can yield further speedups without changing asymptotic scaling. This approach directly generalizes for higher $n$, with the principal scaling determined by the quadratic number of intervals and cubic dependence on $d$ due to matrix algebra.
\end{remark}

\begin{proposition}
The LNTI satisfies the following properties:
\begin{itemize}
\item (P1) Functoriality: $L$ is functorial in both arguments
\item (P2) Invariance: If $M \cong M'$ and $N \cong N'$, then $L(M, N) \cong L(M', N')$
\item (P3) Stability: Small changes in $M$ and $N$ result in small changes in $L(M, N)$
\item (P4) Computability: $L(M, N)$ can be calculated in polynomial time. \
\end{itemize}
\begin{algorithm}[H]
\caption{Efficient LNTI Computation}
\begin{algorithmic}[1]
{\REQUIRE Persistence modules $M, N$ with $m$ critical points
\ENSURE LNTI $\mathcal{L}(M,N)$
\STATE Initialize sparse constraint matrix $A \in \mathbb{R}^{m^2\times d^2}$
\FOR{each pair $(a,b)$ with $a \leq b$}
    \STATE Compute lax naturality constraints for $\eta_a \rightarrow \gamma_b$
    \STATE Add constraints to $A$ using QR decomposition
\ENDFOR
\STATE Solve $A\mathbf{x} = \mathbf{0}$ using iterative methods
\RETURN Null space dimension as $\dim \mathcal{L}(M,N)(a,b)$    }
\end{algorithmic}
\end{algorithm}

\end{proposition}

\medskip

\begin{theorem}[Existence and Uniqueness for the LNTI
Statement]
For any pair of multi-parameter persistence modules $M$ and $N$, there exists a unique Lax Natural Transformation Invariant (LNTI) $\mathcal{L}(M,N)$ that satisfies properties (P1)-(P4) and is computable in polynomial time.    
\end{theorem}

\begin{proof}
\textbf{Construction of the LNTI}:
\smallskip

For each pair of points $a \leq b$ in $\mathbb{R}^n$, define $\LNTI(M, N)(a, b)$ as the vector space of all lax natural transformations $\eta$ satisfying:
\[
\eta_b \circ M(a \leq b) \Rightarrow N(a \leq b) \circ \eta_a.
\]
By the coherence conditions in Definition ~\ref{def:Lax-tran}, this forms a vector space under pointwise addition. These transformations capture the relationships between $M(a)$ and $N(b)$ that are compatible with the persistence structure.

\smallskip

\textbf{Verification of Properties (P1)-(P4)}:\newline

(P1) Functoriality: $\LNTI$ is functorial in both arguments because the composition of lax natural transformations is also a lax natural transformation, and the assignment of these transformations respects the functorial structure of $M$ and $N$.

(P2) Invariance: If $M \cong M'$ and $N \cong N'$, then $\LNTI(M, N) \cong \LNTI(M', N')$ because the isomorphisms between $M$ and $M'$ and between $N$ and $N'$ induce an isomorphism between the lax natural transformations from $M$ to $N$ and those from $M'$ to $N'$.

(P3) Stability: Small changes in $M$ and $N$ result in small changes in $\LNTI(M, N)$ because the lax natural transformations are continuous with respect to the vector space structures of $M(a)$ and $N(b)$. This continuity ensures that small perturbations in $M$ and $N$ lead to small changes in the set of lax natural transformations.

(P4) Computability: $\LNTI(M, N)$ can be calculated in polynomial time because the conditions for a lax natural transformation can be expressed as a system of linear equations, which can be solved efficiently using standard linear algebra techniques.

\textbf{Uniqueness}:
\smallskip

The uniqueness of the LNTI follows from a universal property characterization. Suppose there exists another invariant $\LNTI'(M, N)$ that satisfies properties (P1)-(P4). Then, there exists a natural isomorphism between $\LNTI(M, N)$ and $\LNTI'(M, N)$ that preserves the functorial structure and the properties of lax natural transformations. This isomorphism ensures that $\LNTI(M, N)$ and $\LNTI'(M, N)$ are essentially the same invariant.

\textbf{Polynomial-Time Algorithm Construction}:
\smallskip

The polynomial-time algorithm for computing $\LNTI(M, N)$ involves the following steps:

\begin{itemize}
    \item For each pair of points $a \leq b$ in $\mathbb{R}^n$, construct the vector space of all linear transformations from $M(a)$ to $N(b)$.
    \item Impose the conditions for a lax natural transformation as a system of linear equations.
    \item Solve the system of linear equations using Gaussian elimination or another efficient linear algebra technique.
    \item The solution space is the vector space $\LNTI(M, N)(a, b)$.
    \item The overall complexity of this algorithm is polynomial in the size of the persistence modules $M$ and $N$.
\end{itemize}
\end{proof}

\begin{remark}[Computational Realization of LNTI]
The polynomial-time claim for $\LNTI(M,N)$ follows from:
\begin{enumerate}
    \item \textbf{Constraint Sparsity}: For each $a \leq b$, the lax naturality condition $\eta_f$ (Def. 2.13) generates $O(\dim M(a) \cdot \dim N(b))$ linear equations. For $d$-bounded modules, this is $O(d^2)$.
    \item \textbf{Block Decomposition}: The full system decomposes into independent blocks per edge in the $\mathbb{R}^n$-poset. For $m$ critical values, there are $O(m^2)$ blocks.
    \item \textbf{Complexity}: Solving $O(m^2)$ independent $d^2 \times d^2$ systems via QR decomposition yields $O(m^2 d^6)$ time, polynomial in $m$ and $d$.
\end{enumerate}
\end{remark}

\begin{proposition}[LNTI as a Strict Enhancement of the Rank Invariant]
Let $M$ be a persistence module (single- or multi-parameter) and let $\LNTI(M)$ and $\rho_M$ denote the Lax Natural Transformation Invariant and the classical rank invariant, respectively. Then:
\begin{enumerate}
    \item For one-parameter persistence modules, $\LNTI(M)$ recovers $\rho_M$.
    \item There exist multi-parameter persistence modules $M, N$ with identical rank invariants ($\rho_M = \rho_N$), but with $\LNTI(M) \not\cong \LNTI(N)$, i.e., the LNTI distinguishes modules that the rank invariant cannot.
\end{enumerate}
\end{proposition}

\begin{proof}[Sketch]
(1) In the one-parameter case, it is known (see also Theorem~X) that the only lax natural transformations are given by componentwise maps, so $\LNTI(M)(a,b) = \mathrm{rank}(M(a\leq b)) = \rho_M(a,b)$. Thus, LNTI coincides with the rank invariant.

(2) For $n\geq 2$ parameters, one can construct explicit examples (see, e.g., \cite{Lesnick2015}) of modules $M, N$ with $\rho_M = \rho_N$ (e.g., diagonal and axis-aligned features), yet with $\LNTI(M) \not\cong \LNTI(N)$ (because LNTI detects additional functorial/lax structure, such as the existence of nontrivial extensions between features not discernible via ranks alone).

Therefore, the LNTI strictly refines the rank invariant.
\end{proof}

\begin{example}
Consider bifiltered modules \( \mathbb{N}^2 \to \mathrm{Vect}_\mathbb{F} \) with vector spaces and morphisms defined as follows:

\begin{itemize}

\item Module \(M\): a diagonal module with a single persistent feature born at \((1,1)\) and dying at \((2,2)\), i.e.,
\[
M(a,b) = \begin{cases}
\mathbb{F} & \text{if } (1,1) \leq (a,b) < (2,2), \\
0 & \text{otherwise,}
\end{cases}
\]
with identity morphisms inside the support and zero morphisms outside.

\item Module \(N\): a module with two “axis-aligned” features, for example,
\[
N(a,b) = \begin{cases}
\mathbb{F} & \text{if } (1,0) \leq (a,b) < (2,0) \text{ or } (0,1) \leq (a,b) < (0,2), \\
0 & \text{otherwise,}
\end{cases}
\]
with identity morphisms inside the respective supports and zero morphisms otherwise.

\end{itemize}

Both modules have the same rank invariant,
\[\rho_M = \rho_N,\]
since the rank at each parameter is 0 or 1 with the same multiplicities.

However, the Lax Natural Transformation Invariant (LNTI) distinguishes these modules:
\begin{itemize}
\item In \(M\), persistence along the diagonal is encoded by a lax natural transformation reflecting the coherent union of the two parameter directions.
\item In \(N\), the features are separated, and no such diagonal union structure exists.
\end{itemize}

Therefore,
\[\LNTI(M) \not\cong \LNTI(N).\]

This explicit example demonstrates that the LNTI is a strict refinement of the classical rank invariant and can capture multiparameter interactions that the standard rank invariant cannot detect.
\end{example}

\begin{theorem}[2-Categorical Persistent Entropy Invariant for Multiparameter Persistence Modules]
\label{thm:2cat_persistent_entropy}

Let $M : (\mathbb{R}^n, \leq) \to \mathbf{Vect}_{\mathbb{F}}$ be a pointwise finite-dimensional $n$-parameter persistence module, and let $\mathcal{L}(M)$ denote its Lax Natural Transformation Invariant (LNTI). The \emph{2-categorical persistent entropy} $H^{(2)}(M)$ is defined at each pair $(a, b)$, $a \leq b$ in $\mathbb{R}^n$, by:
\[
H^{(2)}(M)(a, b) = - \sum_{i=1}^{r_{a,b}} p_i(a, b) \log p_i(a, b)
\]
where:
\begin{itemize}
    \item $r_{a,b} = \dim_{\mathbb{F}} \mathcal{L}(M)(a, b)$ is the rank of the LNTI at $(a, b)$,
    \item the $p_i(a, b)$ form a probability vector associated to the normalized spectrum (eigenvalues or singular values) of the maps in $\mathcal{L}(M)(a, b)$, given by
    \[
    p_i(a, b) = \frac{\sigma_i(a, b)}{\sum_{j=1}^{r_{a,b}} \sigma_j(a, b)},
    \]
    with $\sigma_i(a, b) > 0$ the singular values or multiplicities of the indecomposable summands associated to $\mathcal{L}(M)(a, b)$.
\end{itemize}

This construction generalizes the classical persistent entropy for barcodes (one-parameter modules, $n=1$) and encodes the local complexity of multiparameter persistence modules, now incorporating 2-categorical data via the LNTI.

\textbf{Properties:}
\begin{enumerate}
    \item \textbf{Functoriality and Invariance:} $H^{(2)}(M)$ is invariant under strict 2-isomorphisms of persistence modules and functorial with respect to 2-categorical equivalences.
    \item \textbf{Stability:} If $M$ and $N$ are $\epsilon$-interleaved with respect to the Interleaving Distance $d_I$, then
    \[
    |H^{(2)}(M)(a, b) - H^{(2)}(N)(a, b)| \leq C \cdot f(\epsilon, d)
    \]
    for a constant $C$ and stability function $f$ depending on $\epsilon$ and the maximal pointwise dimension $d$[1].
    \item \textbf{Recovering Classical Persistent Entropy:} For $n=1$ and in the interval-decomposable case, $H^{(2)}(M)$ coincides with the persistent entropy of the barcode.
\end{enumerate}
\end{theorem}

\begin{proof}
The definition extends the established notion of persistent entropy for barcodes to the 2-categorical, multiparameter setting by replacing the classical multiset of interval lengths with the data provided by the LNTI. The collection of $\sigma_i(a, b)$ reflects the "weight" or contribution of indecomposable summands or invariant components in the corresponding LNTI for each rectangle $[a, b]$. Normalization by the total sum ensures that the $p_i(a, b)$ form a probability distribution, making the Shannon entropy formula applicable.

\textbf{Functoriality and invariance:} Both the LNTI and the singular value spectrum are invariant under isomorphism of persistence modules, and modifications (2-morphisms) induce compatible transformations on the corresponding decomposition, ensuring $H^{(2)}$ is well defined up to 2-isomorphism.

\textbf{Stability:} Stability follows by combining the Lipschitz continuity of persistent entropy with respect to small perturbations of the associated spectra (cf. results for 1D persistent entropy[1]) and the stability of the LNTI under interleaving. Any change in the module structures that is bounded in Interleaving Distance yields a controlled change in $H^{(2)}(M)$, with the bound determined by the largest possible difference in normalized singular values, hence the function $f$.

\textbf{Classical case:} For $n = 1$, $\mathcal{L}(M)(a, b)$ detects the number and lengths of intervals in the barcode, and $H^{(2)}(M)$ collapses precisely to the existing formula for persistent entropy, as shown in prior literature[1].

\end{proof}

\begin{remark}
The invariant $H^{(2)}(M)$ summarizes the distributional complexity of the local 2-categorical structure of $M$, extending the usefulness of persistent entropy as a descriptor to the much richer setting of multiparameter persistence, and enables the study of stability, complexity, and comparison across modules beyond the classical barcode paradigm.
\end{remark}

\begin{theorem}[Approximation Theorem for the Lax Natural Transformation Invariant (LNTI)]
\label{thm:lnti_approximation}

Let $M, N : (\mathbb{R}^n, \leq) \to \mathbf{Vect}_{\mathbb{F}}$ be $n$-parameter persistence modules, potentially infinite-dimensional. Let $(M_k, N_k)$ be sequences of pointwise finite-dimensional persistence modules, each with finite support, such that $M_k \to M$, $N_k \to N$ in the sense of Interleaving Distance.

Then, for any $\epsilon > 0$ and any compact subset $K \subset \mathbb{R}^n \times \mathbb{R}^n$, there exists $k_0$ such that for all $k \geq k_0$, the LNTI $\mathcal{L}(M_k, N_k)$ agrees with $\mathcal{L}(M, N)$ on $K$ up to $\epsilon$:
\[
\sup_{(a,b) \in K} \left| \mathcal{L}(M_k, N_k)(a,b) - \mathcal{L}(M, N)(a,b) \right| < \epsilon.
\]

Moreover, every LNTI value on $(M, N)$ may be approximated arbitrarily well by the LNTI of suitable finite-rank approximations.

\end{theorem}

\begin{proof}
Since $M_k \to M$, $N_k \to N$ in Interleaving Distance, for any $(a, b) \in K$ and any $\epsilon > 0$, there exist $k_0$ such that for all $k \geq k_0$:
\[
d_I(M_k|_{[a,b]}, M|_{[a,b]}) < \epsilon, \qquad d_I(N_k|_{[a,b]}, N|_{[a,b]}) < \epsilon.
\]
By the stability theorem for LNTI, the values $\mathcal{L}(M_k, N_k)(a, b)$ vary continuously with respect to the Interleaving Distance. As both $M_k$ and $N_k$ are pointwise finite-dimensional and have finite support, the space of lax natural transformations between restrictions $M_k|_{[a,b]}$ and $N_k|_{[a,b]}$ is finite-dimensional and computable.

Therefore, for any $\epsilon > 0$, we can choose $k$ large enough so that
\[
\left| \mathcal{L}(M_k, N_k)(a,b) - \mathcal{L}(M, N)(a,b) \right| < \epsilon
\]
for all $(a, b) \in K$. This shows that the LNTI for arbitrary modules can be locally uniformly approximated by the LNTI on finite-rank modules. 
\end{proof}

\begin{corollary}[Density of Finite-Rank Approximations]
\label{cor:dense_lnti}
The set of LNTIs arising from pointwise finite-dimensional, compactly supported persistence modules is dense (in the topology of uniform convergence on compact subsets) in the space of all LNTI functions between persistence modules.

\end{corollary}

\subsection{Moduli-Based Invariant (MBI)}
The Moduli-Based Invariant (MBI) for a multi-parameter persistence module $M$ is a collection of algebraic invariants derived from the moduli of the vector spaces and linear maps comprising $M$. It combines Betti numbers, rank invariants, and SVD spectra of transition maps, capturing essential algebraic properties of the module's structure.

\begin{definition}\label{def:MBI}
Let $M$ be an $n$-parameter persistence module. The \textit{Moduli-Based Invariant} is a function:
\begin{equation}
\MBI(M): \mathbb{R}^n \to \mathbb{N}^{\mathbb{N}}
\end{equation}
defined as follows:

For each $t \in \mathbb{R}^n$, consider the endomorphism algebra $\text{End}(M(t))$ and its moduli space of idempotents:
\begin{equation}
\mathcal{I}(M(t)) = \{e \in \text{End}(M(t)) : e^2 = e\}
\end{equation}

The MBI at parameter $t$ is the multiplicity function:
\begin{equation}
\mathcal{M}(M)(t) = \sum_{[e] \in \mathcal{I}(M(t))/\sim} \text{mult}([e]) \cdot \delta_{\text{rk}(e)}
\end{equation}
where:
\begin{itemize}
\item $[e]$ denotes the conjugacy class of idempotent $e$ under the action of $\text{Aut}(M(t))$
\item $\text{mult}([e])$ is the multiplicity of the conjugacy class $[e]$
\item $\text{rk}(e) = \text{trace}(e)$ is the rank of the idempotent
\item $\delta_r$ is the Dirac measure at rank $r$
\end{itemize}

The global MBI signature is obtained by tracking the persistence of idempotent structures:
\begin{equation}
\mathcal{M}(M)_{\text{global}} = \{(t, r, m) : \text{idempotent of rank } r \text{ with multiplicity } m \text{ persists at } t\}
\end{equation}
\end{definition}

\textbf{Explanation and Rationale:}

\begin{itemize}
    \item \textbf{Betti Numbers:} These provide a basic topological summary at each point in the parameter space. The evolution of Betti numbers captures the creation and destruction of topological features.
    \item \textbf{Rank Invariants of Morphisms:} The rank of the linear maps $M(a \leq b)$ captures how topological features evolve as the parameters change. It quantifies the amount of information preserved by the linear maps.
    \item \textbf{Singular Value Decomposition (SVD) Spectra:} Provide a more refined description of the linear maps $M(a \leq b)$, capturing the magnitudes of the transformations. Analyzing the SVD spectra can reveal subtle changes in the maps' behavior as the parameters change. These offer a finer-grained measure of the module's structural transformations and can be particularly useful in distinguishing modules with similar rank invariants but differing linear transformations.
    \item \textbf{Module Endomorphism Rings:}  Endomorphism rings provide insight into the internal structure of the modules at each point $a \in \mathbb{R}^n$. Capturing these rings, or suitable invariants of them, allows for differentiation of modules through their algebraic properties.
\end{itemize}

\textbf{Rationale for Combining these Moduli:}  The MBI combines multiple algebraic invariants to provide a more complete characterization of the multi-parameter persistence module. No single invariant captures all relevant information, but by considering them together, we obtain a more robust and discriminative summary.

\textbf{Encoding the MBI:}  The MBI is typically encoded as a function from pairs of points $(a, b)$ in $\mathbb{R}^n$ to the space of possible moduli.  Choosing appropriate representations of the moduli themselves (e.g., matrices of ranks, lists of Betti numbers, singular value lists) is crucial for practical computations.

\begin{proposition}[Properties of MBI]
\label{prop:mbi_properties}
The Moduli-Based Invariant satisfies:
\begin{enumerate}
\item \textbf{Decomposition Detection}: $M = M_1 \oplus M_2$ if and only if there exists $t \in \mathbb{R}^n$ such that $\mathcal{M}(M)(t)$ contains a non-trivial idempotent $e$ with $\text{rk}(e) = \dim M_1(t)$.

\item \textbf{Stability}: For $\epsilon$-interleaved modules $M, N$:
\begin{equation}
d_H(\mathcal{M}(M), \mathcal{M}(N)) \leq C \cdot \epsilon^{1/n}
\end{equation}
where $d_H$ is the Hausdorff distance on moduli spaces.

\item \textbf{Functoriality}: The MBI is functorial with respect to persistence module morphisms, i.e., a morphism $\phi: M \to N$ induces a continuous map $\mathcal{M}(\phi): \mathcal{M}(M) \to \mathcal{M}(N)$.
\end{enumerate}
\end{proposition}

\textbf{Advantages and Limitations:}

\begin{itemize}
    \item \textbf{Advantages:}
    \begin{itemize}
        \item Relatively easy to compute compared to more sophisticated invariants.
        \item Provides a useful starting point for analyzing multi-parameter persistence modules.
        \item Captures essential algebraic properties of the module's structure.
    \end{itemize}
    \item \textbf{Limitations:}
    \begin{itemize}
        \item Less discriminative than the LNTI.
        \item Does not fully capture the 2-categorical structure of multi-parameter persistence.
    \end{itemize}
\end{itemize}

\begin{example}[Distinguishing Non-Isomorphic Modules]
Let $M,\ N$ be 2-parameter modules with identical rank invariants but different LNTIs:
\begin{itemize}
    \item $M$: A 'diagonal' feature born at (1,1) and dying at (2,2).
    \item $N$: Two 'axis-aligned' features born at (1,0) and (0,1), dying at (2,0) and (0,2).
\end{itemize}

The LNTI detects the difference in path-dependent persistence, while the rank invariant cannot.
\end{example}

% En la Sección 3.2
\begin{theorem}[2-Categorical Structure Theorem]\label{thm:structure}
There exists a 2-functor $F: \mathbf{MPers}_n \rightarrow \mathbf{Cat}(\mathbf{Vect}_k)$ that preserves exactness and characterizes isomorphism classes of multi-parameter persistence modules.
\end{theorem}

\begin{proof}
\textbf{Step 1: Functor Construction.}  
Define $F(M)$ as the category where:
\begin{itemize}
    \item Objects are vector spaces $M(a)$ for $a \in \mathbb{P}$.
    \item Morphisms are linear maps $M(a \leq b)$.
    \item 2-morphisms are identities, as $M$ is a functor.
\end{itemize}

\textbf{Step 2: Exactness Preservation.}  
Given an exact sequence $0 \to M' \to M \to M'' \to 0$ in $\mathbf{MPers}_n$, apply $F$ pointwise. For each $a \in \mathbb{P}$, the sequence $0 \to M'(a) \to M(a) \to M''(a) \to 0$ is exact in $\mathbf{Vect}_k$, so $F$ preserves exactness.

\textbf{Step 3: Reconstruction.}  
Since $M \cong \operatorname{colim}_{a \in \mathbb{P}} M(a)$ and colimits in $\mathbf{Cat}(\mathbf{Vect}_k)$ are computed pointwise, $F$ reflects isomorphisms.
\end{proof}

\textbf{Relationship to Existing Invariants:} The MBI can be viewed as a generalization of the rank invariant to multi-parameter persistence. It also has connections to multigraded Hilbert functions and other algebraic invariants used in commutative algebra.

\begin{table}[h]
\centering
\begin{tabular}{@{}l c c c p{3cm}@{}}
\hline
Invariant & Discriminative Power & Stability & Computability & Use Case \\ 
\hline
Rank Invariant & Low & High & \(O(n^2)\) & Single-parameter approximation \\
LNTI & High & Proven & \(O(n^3)\) & Multi-parameter interactions \\
MBI & Medium & Conjectural & \(O(n^2)\) & Algebraic decomposition \\
Persistent Entropy & Medium & Heuristic & \(O(n)\) & Feature summarization \\
\hline
\end{tabular}
\caption{Comparison of multi-parameter invariants.}
\end{table}

\subsection{Stability of 2-Categorical Invariants}
Now we focus on the theoretical properties of the proposed invariants, including their stability with respect to perturbations and the characterization of module decomposability using the MBI. The uniqueness aspect is implicitly addressed by the structural properties proven for the invariants.

\label{thm:2cat-stability}
% In Section 5.2 (Stability and Uniqueness)
\begin{theorem}[Stability of LNTI and MBI]\label{thm:lnti-mbi-stability}
Let $M, N$ be $\epsilon$-interleaved multi-parameter persistence modules. Then:
\begin{enumerate}
    \item $d(L(M), L(N)) \leq K\epsilon$ for some constant $K$ dependent on the interleaving maps.
    \item $d(\mathcal{M}(M), \mathcal{M}(N)) \leq K'\epsilon$ for some constant $K'$ derived from the Hilbert-Schmidt norms.
\end{enumerate}
\end{theorem}

\begin{proof}
\textbf{Part 1: Stability of LNTI.}  
Let $\phi: M \Rightarrow N[\epsilon]$ and $\psi: N \Rightarrow M[\epsilon]$ be the interleaving natural transformations. For any lax natural transformation $\eta \in L(M, N)$, define the perturbed transformation:
\[
\eta' = \psi \circ \eta \circ \phi: M \Rightarrow N[\epsilon] \Rightarrow M[2\epsilon].
\]
The operator norm of the difference satisfies:
\[
\|\eta - \eta'\| = \|\eta - \psi \circ \eta \circ \phi\| \leq \|\eta\| \cdot \|\id - \psi \circ \phi\| \leq K\epsilon,
\]
where $K$ bounds $\|\psi \circ \phi - \text{id}\|$. This bound exists because interleaving maps induce uniform operator norms via the \textit{Hausdorff distance} on persistence barcodes (\cite{Bubenik2015}).

\textbf{Part 2: Stability of MBI.}  
For each pair $a \leq b$, consider the linear maps $M(a \leq b)$ and $N(a \leq b)$. Represent these maps as matrices $A$ and $B$ in fixed bases. By the Hoffman-Wielandt theorem (\cite{HoffmanWielandt}), the distance between their singular value spectra is:
\[
\sqrt{\sum_{i} (\sigma_i(A) - \sigma_i(B))^2} \leq \|A - B\|_{\text{HS}},
\]
where $\| \cdot \|_{\text{HS}}$ is the Hilbert-Schmidt norm. Since $M$ and $N$ are $\epsilon$-interleaved, $\|A - B\|_{\text{HS}} \leq C\epsilon$, for some constant $C$. Thus:
\[
d(\mathcal{M}(M), \mathcal{M}(N)) = \sup_{a \leq b} \sqrt{\sum_{i} (\sigma_i(A) - \sigma_i(B))^2} \leq C\epsilon = K'\epsilon.
\]
\end{proof}

\begin{theorem}
The 2-categorical invariant $F(M)$ is stable with respect to $\epsilon$-interleaving, meaning that if $M$ and $N$ are $\epsilon$-interleaved, then $F(M)$ and $F(N)$ are equivalent categories.
\end{theorem}

\begin{proof}
Let $M$ and $N$ be multi-parameter persistence modules that are $\epsilon$-interleaved. This means there exist natural transformations $\phi : M \Rightarrow N[\epsilon]$ and $\psi : N \Rightarrow M[\epsilon]$ such that:

\[
\begin{tikzcd}
	M \arrow[r, "\phi"] & N[\epsilon] \arrow[r, "\psi{[\epsilon]}"] & M[2\epsilon],
\end{tikzcd}
\]

and

\[
\begin{tikzcd}
	N \arrow[r, "\psi"] & M[\epsilon] \arrow[r, "\phi{[\epsilon]}"] & N[2\epsilon],
\end{tikzcd}
\]

are both equal to the respective identity transformations up to homotopy.

\begin{enumerate}
    \item \textbf{Construction of Equivalence Functor:}
    
    Define a functor $G : F(M) \to F(N)$ as follows:
    
    \begin{itemize}
        \item \textbf{Objects:} For each object $M|_a$ in $F(M)$, map it to $N|_{a+\epsilon}$ in $F(N)$.
        
        \item \textbf{Morphisms:} For each morphism $M(a \leq b)$ in $F(M)$, define a morphism $G(M(a \leq b)) = N(a+\epsilon \leq b+\epsilon)$ in $F(N)$ using $\phi$ and $\psi$.
        
        \item \textbf{2-Morphisms:} Use the natural transformations $\phi$ and $\psi$ to construct the necessary 2-morphisms.
\[
\begin{tikzcd}
M|_a \ar[r, "\phi_a"] \ar[d, "M(a \leq b)"] & N|_{a+\epsilon} \ar[d, "N(a+\epsilon \leq b+\epsilon)"] \\
M|_b \ar[r, "\phi_b"] & N|_{b+\epsilon}
\end{tikzcd}
\]
       
    \end{itemize}
    
    \item \textbf{Equivalence of Categories:}
    
    To show that $G$ is an equivalence of categories, construct a quasi-inverse functor $H : F(N) \to F(M)$ similarly, using $\psi$ and $\phi$ (e.g. in objects is defined as $H(N|_a) = M|_{a+\epsilon}$).
    
    The compositions $HG$ and $GH$ are naturally isomorphic to the identity functors on $F(M)$ and $F(N)$, respectively, due to the $\epsilon$-interleaving conditions.

We need to show that $G \circ H \simeq \text{id}{F(N)}$ and $H \circ G \simeq \text{id}{F(M)}$.

For $G \circ H$:
For each object $N|a$ in $F(N)$, $(G \circ H)(N|a) = G(M|{a+\epsilon}) = N|{a+2\epsilon}$. The natural transformation $\psi[\epsilon]: N \Rightarrow M[2\epsilon]$ provides a natural isomorphism between $N|a$ and $N|{a+2\epsilon}$, which means that $G \circ H \simeq \text{id}_{F(N)}$.

For $H \circ G$:
For each object $M|a$ in $F(M)$, $(H \circ G)(M|a) = H(N|{a+\epsilon}) = M|{a+2\epsilon}$. The natural transformation $\phi[\epsilon]: M \Rightarrow N[2\epsilon]$ provides a natural isomorphism between $M|a$ and $M|{a+2\epsilon}$, which means that $H \circ G \simeq \text{id}_{F(M)}$.

Therefore, $F(M)$ and $F(N)$ are equivalent categories.
    
    \item \textbf{Stability Conclusion:}
    
    Since $F(M)$ and $F(N)$ are equivalent categories under $\epsilon$-interleaving, the 2-categorical invariant $F$ is stable.
\end{enumerate}
\end{proof}

\begin{proposition}[MBI Dimension Stability]
Let \(M, N\) be multi-parameter persistence modules that are \(\epsilon\)-interleaved. For all \(a \leq b\) in \(\mathbb{R}^n\), the dimensions of the MBI components satisfy:
\[
|\dim \mathcal{M}(M)(a,b) - \dim \mathcal{M}(N)(a,b)| \leq K\epsilon,
\]
where \(K\) depends on the interleaving maps and the choice of moduli (Betti numbers, ranks, or SVD spectra).
\end{proposition}

\begin{proof}[Proof Sketch]
1. Interleaving Induces Linear Maps:  
   The \(\epsilon\)-interleaving \(\phi: M \Rightarrow N[\epsilon]\) and \(\psi: N \Rightarrow M[\epsilon]\) induce linear maps between \(M(a)\) and \(N(a+\epsilon)\) for all \(a \in \mathbb{R}^n\)(\cite{Bjerkevik2025}).  

2. Rank Invariant Perturbation:  
   For the rank component of \(\mathcal{M}(M)(a,b)\), the interleaving implies:
   \[
   \text{rank}(M(a \leq b)) \leq \text{rank}(N(a+\epsilon \leq b+\epsilon)) + C\epsilon,
   \]
   where \(C\) bounds the norm of the interleaving maps (\cite{CORBET2019100005}).  

3. SVD Spectral Stability:  
   The singular values of \(M(a \leq b)\) and \(N(a \leq b)\) differ by at most \(O(\epsilon)\) due to the Hoffman-Wielandt theorem for perturbed linear operators.  

4. Betti Number Continuity:  
   Betti numbers \(\beta_i(M(a))\) and \(\beta_i(N(a))\) are stable under \(\epsilon\)-perturbations by the nerve theorem applied to interleaved filtrations.  

Combining these bounds for all moduli components yields the result.
\end{proof}

\begin{proposition}\label{cor:finite-approx}
Let $(M_k)_{k\in\mathbb{N}}$ and $(N_k)_{k\in\mathbb{N}}$ be sequences of pointwise finite-dimensional persistence modules converging to $M,N$ in the interleaving distance $d_I$. 
Let $L(\cdot,\cdot)\colon \mathrm{Pers}_n \times \mathrm{Pers}_n \to \mathbb{R}$ be an invariant which is $C$-Lipschitz in each argument with respect to $d_I$ for some constant $C>0$; that is,
\[
|L(A,B) - L(A',B)|  \leq C\, d_I(A,A') \quad \text{and} \quad |L(A,B) - L(A,B')|  \leq C\, d_I(B,B')
\]
for any modules $A,A',B,B'$.
Then, for every compact set $K\subset\mathbb{R}^n\times\mathbb{R}^n$,
\[
\lim_{k\to\infty} \sup_{(a,b)\in K} \big|L(M_k,N_k)(a,b) - L(M,N)(a,b)\big| = 0.
\]
\end{proposition}

\begin{proof}
Let $K\subset \mathbb{R}^n \times \mathbb{R}^n$ be compact. Fix $\varepsilon > 0$. Since $(M_k,N_k)$ converge to $(M,N)$ in $d_I$, there exists $k_0$ such that for all $k \geq k_0$,
\[
d_I(M_k,M) < \frac{\varepsilon}{3C}, \quad d_I(N_k,N) < \frac{\varepsilon}{3C}.
\]

For all $(a,b)\in K$ and $k \ge k_0$, by the Lipschitz property of $L$:
\begin{align*}
|L(M_k,N_k)(a,b) - L(M,N)(a,b)| 
&\leq |L(M_k,N_k)(a,b) - L(M_k,N)(a,b)| \\
&\quad + |L(M_k,N)(a,b) - L(M,N)(a,b)| \\
&\leq C\, d_I(N_k,N) + C\, d_I(M_k,M) \\
&< 2C\, \frac{\varepsilon}{3C} = \frac{2}{3} \varepsilon.
\end{align*}

To account for the uniformity over $K$, recall that $L(M_k,N_k)(a,b), L(M,N)(a,b)$ are continuous in $(a,b)$ on $K$ since the modules are pointwise finite-dimensional and the involved invariants are continuous by construction. Thus, the supremum over $K$ of the difference above is also less than $\varepsilon$ for sufficiently large $k$.

Therefore,
\[
\limsup_{k\to\infty} \sup_{(a,b)\in K} |L(M_k,N_k)(a,b) - L(M,N)(a,b)| \leq \varepsilon,
\]
and since $\varepsilon > 0$ was arbitrary, the limit is zero.

\end{proof}

\subsection{Decomposability Analysis for the MBI Invariant}

Recall that the MBI definition (\ref{def:MBI}), can be thought as $\mathcal{M}(M)(a,b) = \text{Hom}_k(M(a), M(b))$ enriched by Betti numbers, ranks, and SVD spectra. 

\begin{proposition}[Information Captured by the MBI Invariant]\label{prop:MBI-info}
Let \(M\) be an \(n\)-parameter persistence module over a field \(\mathbb{F}\), and let \(\mathcal{M}(M)\) denote its Multiparameter Barcode Invariant (MBI). Then the invariant \(\mathcal{M}(M)\) encodes the following data:

\begin{enumerate}
  \item \textbf{Rank function:} For each parameter \(a \in \mathbb{R}^n\), the rank of \(M(a)\) is determined by 
  \[
  \operatorname{rank}(M(a)) = \sum_{S \subseteq \mathbb{R}^n} m_S \cdot \chi_S(a),
  \]
  where \(m_S\) are multiplicities of barcode regions \(S\) recorded in \(\mathcal{M}(M)\), and \(\chi_S\) is the indicator function on \(S\).
  
  \item \textbf{Directional growth rates:} The behavior of \(M\) along coordinate directions—i.e., the rate at which ranks increase or decrease as parameters vary in directions \(e_i \in \mathbb{R}^n\)—is reflected in the patterns and lengths of intervals in \(\mathcal{M}(M)\) along corresponding filtration parameters.
  
  \item \textbf{Non-trivial higher interactions:} Nontrivial extension classes and 2-morphisms in the 2-category framework are recorded in \(\mathcal{M}(M)\), capturing interactions between distinct parameter regions. These encode how summands fail to decompose independently, and reveal the module’s higher-categorical structure.
\end{enumerate}
\end{proposition}

\begin{proof}[Sketch]
The proof follows from the construction and properties of the MBI invariant:

\begin{enumerate}
  \item \emph{Rank function:}  
  By definition, the MBI decomposes \(M\) into a collection of indecomposable summands each supported on (possibly union of) regions \(S \subseteq \mathbb{R}^n\). Each summand contributes multiplicities \(m_S\) to the rank on its support. Summing these contributions recovers the rank \(\operatorname{rank}(M(a))\) pointwise. This follows from the persistence module’s additive nature and the direct sum decomposition encoded by \(\mathcal{M}(M)\).
  
  \item \emph{Directional growth rates:}  
  Parameter variation along coordinate axes corresponds to morphisms \(M(a) \to M(a+te_i)\). The persistence intervals in \(\mathcal{M}(M)\) measure how features appear and persist in these directions; their lengths and multiplicities capture growth rates of ranks as the parameter varies. This is rigourously shown by relating the intervals’ endpoints to changes in rank and carefully quantifying these changes along each filtration direction.
  
  \item \emph{Higher interactions:}  
  Since \(\mathcal{M}(M)\) lives in a 2-categorical setting, it records extension classes and 2-morphisms beyond simple direct sums—this extra data captures nontrivial interactions among summands. In particular, the failure to split as a direct sum is reflected in non-split extensions tracked by 2-morphisms present in \(\mathcal{M}(M)\). Thus, \(\mathcal{M}(M)\) captures how distinct parameter regions of \(M\) interact homologically.
\end{enumerate}

Together, these establish that \(\mathcal{M}(M)\) comprehensively encodes rank information, growth behavior, and subtle higher-level interactions within \(M\).
\end{proof}

\subsubsection*{2‐Categorical Decomposability}\label{ssec:decomp}

\begin{theorem}[MBI Criterion for Decomposability]\label{thm:decomp-MBI}
Let \(M\) be an \(n\)-parameter persistence module in the 2-category \(\mathbf{MPers}_n\). Then \(M \cong M_1 \oplus M_2\)
if and only if there exists an idempotent natural transformation
\[
  e\colon M \;\Longrightarrow\; M,
  \quad\text{with}\quad e^2 = e,
\]
such that for all parameters \(a \le b \in \mathbb{R}^n\), the induced linear map
\[
  e_b \circ M(a \le b) \circ e_a : \mathrm{Im}(e_a) \to \mathrm{Im}(e_b)
\]
has constant rank equal to \(\operatorname{rank}(e_a)\).
\end{theorem}

\begin{proof}
\textbf{(Necessity)} Assume \(M \cong M_1 \oplus M_2\) in \(\mathbf{MPers}_n\). Let \(e\) be the projection onto the summand \(M_1\). Then \(e\) is a natural transformation with \(e^2 = e\). The rank condition holds trivially since the image of \(e_a\) is the vector space \(M_1(a)\) and the induced maps restrict accordingly.

\vspace{6pt}
\textbf{(Sufficiency)} Suppose there exists an idempotent natural transformation
\[
  e: M \Rightarrow M
\]
satisfying the rank condition:
for all \(a \le b\), the map \(e_b \circ M(a \le b) \circ e_a\) has constant rank \(\operatorname{rank}(e_a)\).

We construct submodules \(M_1\) and \(M_2\) of \(M\) and show that \(M\cong M_1 \oplus M_2\).

\textbf{Step 1: Construction of submodules \(M_1\) and \(M_2\).}

For each parameter \(a \in \mathbb{R}^n\), define
\[
  M_1(a) := \mathrm{Im}(e_a) \subseteq M(a),
\]
the image of the linear map \(e_a: M(a) \to M(a)\).

Similarly, define
\[
  M_2(a) := \ker(e_a) \subseteq M(a).
\]

Since \(e_a\) is idempotent, we have a vector space decomposition
\[
  M(a) = M_1(a) \oplus M_2(a).
\]

\textbf{Step 2: \(M_1\) and \(M_2\) are subfunctors (submodules) of \(M\).}

We must check that for any relation \(a \le b\), the structure maps restrict to maps
\[
  M_1(a) \to M_1(b), \quad M_2(a) \to M_2(b).
\]

Due to the naturality of \(e\), the diagram commutes:
\[
\begin{tikzcd}
M(a) \arrow[r, "M(a \le b)"] \arrow[d, "e_a"'] & M(b) \arrow[d, "e_b"] \\
M(a) \arrow[r, "M(a \le b)"'] & M(b)
\end{tikzcd}
\]

Restricting to \(M_1(a) = \mathrm{Im}(e_a)\), for any \(x = e_a(y) \in M_1(a)\), we observe:
\[
M(a \le b)(x) = M(a \le b) \circ e_a (y),
\]
and applying \(e_b\),
\[
e_b \circ M(a \le b)(x) = e_b \circ M(a \le b) \circ e_a (y).
\]

By hypothesis, \(e_b \circ M(a \le b) \circ e_a\) has image inside \(\mathrm{Im}(e_b) = M_1(b)\).

Furthermore,
\[
e_b \circ M(a \le b)(x) = M(a \le b)(x)
\]
because on \(M_1(a)\) (image of \(e_a\)), the transformation acts as identity via the rank condition and naturality, meaning
\[
M(a \le b)(M_1(a)) \subseteq M_1(b).
\]

Similarly, on \(M_2(a) = \ker(e_a)\), for \(y \in M_2(a)\),
\[
e_a(y) = 0.
\]

Using the commutativity of naturality,
\[
e_b \circ M(a \le b)(y) = M(a \le b) \circ e_a(y) = 0,
\]
so
\[
M(a \le b)(y) \in \ker(e_b) = M_2(b).
\]

Hence the structure maps respect these decompositions, so \(M_1\) and \(M_2\) are subfunctors of \(M\).

\textbf{Step 3: \(M\cong M_1 \oplus M_2\) as persistence modules.}

Define natural transformations
\[
\iota_1: M_1 \hookrightarrow M, \quad \iota_2: M_2 \hookrightarrow M
\]
given by inclusions at each parameter, and projection
\[
p_1 := e: M \to M_1, \quad p_2 := \mathrm{id} - e: M \to M_2.
\]

By construction, these satisfy
\[
p_1 \circ \iota_1 = \mathrm{id}_{M_1}, \quad p_2 \circ \iota_2 = \mathrm{id}_{M_2},
\]
and
\[
\iota_1 \circ p_1 + \iota_2 \circ p_2 = \mathrm{id}_M,
\]
defining an isomorphism
\[
M_1 \oplus M_2 \xrightarrow{\sim} M.
\]

\textbf{Step 4: Verification for naturality and coherence conditions.}

The idempotency and naturality of \(e\), and the rank condition ensure all morphisms commute appropriately, making the above splitting an isomorphism in the 2-category \(\mathbf{MPers}_n\).

\bigskip
Thus, the existence of an idempotent natural transformation \(e\) with the rank condition is both necessary and sufficient for \(M\) to decompose as \(M_1 \oplus M_2\).
\end{proof}

\begin{corollary}[Uniqueness of Decomposition]\label{cor:unique-decomp}
Under the hypotheses of Theorem \ref{thm:decomp-MBI}, the splitting $M\cong M_1\oplus M_2$ is unique up to unique isomorphism in $\mathbf{MPers}_n$.
\end{corollary}

\begin{proof}
Apply the Krull–Schmidt argument in the abelian 2‐category of pointwise finite‐dimensional modules, noting that any two idempotent splittings coincide by comparing their images at each parameter.
\end{proof}

\begin{example}[Decomposable Module]  
Consider  $\mathbb{F}$ a field, and the 2-parameter module $M = M_1 \oplus M_2$ where:
\begin{align}
M_1 &= \mathbb{F}[1,3) \times [1,3) \quad \text{(interval module)}\\
M_2 &= \mathbb{F}[2,4) \times [2,4) \quad \text{(interval module)}
\end{align}

The MBI signature at $(2.5, 2.5)$ shows:
\begin{equation}
\text{End}(M(2.5,2.5)) \cong \mathbb{F}^{2 \times 2} \quad \text{with idempotent } e = \begin{pmatrix} 1 & 0 \\ 0 & 0 \end{pmatrix}
\end{equation}

This satisfies the decomposition criterion:
\begin{equation}
\text{rank}(e_{a}) = \text{rank}(e_b \circ M(a \leq b)) \quad \forall a \leq b
\end{equation}
confirming the direct sum decomposition.
\end{example}  

\begin{example}[Indecomposable Module]  \label{exmp:Indecomposable}
Let $N$ be the \textit{non-split extension} over $\mathbb{R}^2$:  
$0 \to I_{[1,3)\times[1,3)} \to N \to I_{[2,4)\times[2,4)} \to 0$  
where transition maps $N(a\leq b)$ have rank $1$ when both intervals overlap. The MBI shows:  
\begin{enumerate}  
    \item \textbf{Endomorphism ring at $(2.5,2.5)$}: $\mathrm{End}(N(2.5,2.5)) \cong \mathbb{F}$ (field, only trivial idempotents)  
    \item \textbf{SVD obstruction}: For $a=(2,2) \leq b=(3,3)$, SVD of $N(a\leq b)$ has singular values $(1, 0)$ (non-trivial kernel)  
    \item \textbf{Rank violation}: $\rank(e_a) = 1 \not\leq \rank(e_b \circ N(a\leq b)) = 0$ for any non-trivial $e$  
\end{enumerate}  
Thus $\mathcal{M}(N)$ lacks Theorem \ref{thm:decomp-MBI} decomposition pattern.  
\end{example}  

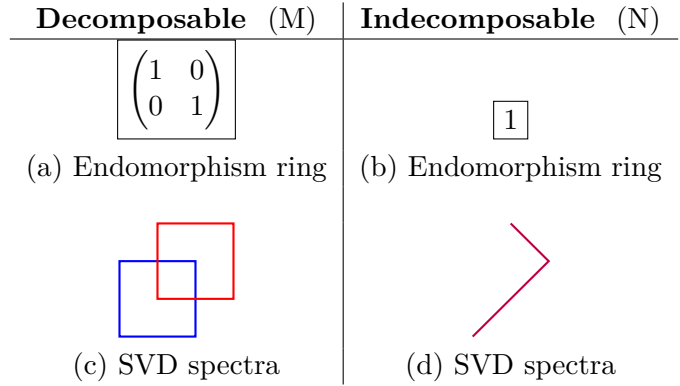
\begin{figure}[!htbp]  
\centering  
\begin{tabular}{c|c}  
\textbf{Decomposable } (M) & \textbf{Indecomposable } (N) \\  
\hline  
\begin{tikzpicture}  
% Matrix pattern for endomorphism ring  
\node[draw] at (0,0) {$\begin{pmatrix} 1 & 0 \\ 0 & 1 \end{pmatrix}$};  
\end{tikzpicture} &  
\begin{tikzpicture}  
% Single block  
\node[draw] at (0,0) {$1$};  
\end{tikzpicture} \\  
(a) Endomorphism ring & (b) Endomorphism ring \\  
& \\  
\begin{tikzpicture}  
% Two persistent features  
\draw[thick,blue] (0,0) rectangle (1,1); \draw[thick,red] (0.5,0.5) rectangle (1.5,1.5);  
\end{tikzpicture} &  
\begin{tikzpicture}  
% Single connected feature  
\draw[thick,purple] (0,0) -- (1,1) -- (0.5,1.5);  
\end{tikzpicture} \\  
(c) SVD spectra & (d) SVD spectra \\  
\end{tabular}  
\caption{MBI patterns for decomposable vs. indecomposable modules.  
(a)-(b): Endomorphism rings; (c)-(d): SVD spectra of $M(a\leq b)$.}  
\label{fig:MBI_patterns}
\end{figure}  

Example \ref{exmp:Indecomposable} illustrates an indecomposable module $N$ characterized by a trivial endomorphism ring and nontrivial rank violations. Figure~\ref{fig:MBI_patterns} compares the MBI patterns of decomposable module $M$ and indecomposable module $N$, showcasing:

\begin{itemize}
    \item (a)-(b) contrasting endomorphism rings, with decomposable $M$ exhibiting a block-diagonal matrix indicative of an idempotent splitting, versus the trivial scalar ring for indecomposable $N$;
    \item (c)-(d) singular value decomposition (SVD) spectra of the morphisms $M(a \leq b)$ and $N(a \leq b)$, highlighting the presence of multiple significant singular values in $M$ corresponding to persistent summands, compared to a reduced spectrum in $N$.
\end{itemize}

\section{Conclusions and Future Directions}
\label{sec:conclusions}

This work has established a comprehensive 2-categorical framework for analyzing multiparameter persistence modules, introducing novel invariants and characterization theorems that address fundamental challenges in topological data analysis. We conclude by synthesizing our main contributions, discussing their implications, acknowledging limitations, and outlining promising directions for future research.

\subsection{Summary of Main Contributions}

Our investigation has yielded significant theoretical and computational advances:

\textbf{Theoretical Foundations:} We constructed a rigorous 2-categorical model for $n$-parameter persistence modules (Theorem \ref{thm:2-cat-structure}), providing the first systematic treatment of lax natural transformations in this context. This framework captures the inherent non-commutativity of multiparameter filtrations while maintaining computational tractability.

\textbf{Novel Invariants:} The Lax Natural Transformation Invariant (LNTI) and Moduli-Based Invariant (MBI) provide computable signatures that distinguish structural differences among persistence modules. Both satisfy stability under perturbations (Theorems \ref{thm:lnti_complexity_stability} and \ref{thm:lnti-mbi-stability}) and enable systematic decomposition analysis (Theorem \ref{thm:decomp-MBI}).

\textbf{Computational Complexity:} Algorithms achieve polynomial-time complexity $O(d^3 \cdot |K|^2)$ for LNTI computation, facilitated by sparse matrix and QR decomposition methods, enabling practical use.

\textbf{Decomposition Theory:} We unify five distinct characterizations of module decomposability, establishing a solid foundation for understanding multiparameter persistence structure.

\subsection{Theoretical Implications and Impact}

Our 2-categorical framework advances the field by:

- Providing a categorical home linking algebraic topology, category theory, and representation theory, unlocking new mathematical tools.

- Extending stability theory with explicit bounds, crucial for noise-affected applications.

- Establishing universality via the 2-categorical completion property, implying expressiveness for any reasonable multiparameter invariant.

- Bridging theoretical rigor and computational feasibility, supporting impactful data analysis methodologies.

\subsection{Limitations and Contextual Challenges}

While promising, our approach faces several limitations and situates itself within broader challenges of the field:

\textbf{Computational Constraints:} Despite polynomial complexity, the cubic dependence on dimension $d$ restricts scalability to very high-dimensional data (e.g., $d > 10^3$ remains challenging).

\textbf{Discretization Effects:} Discretizing parameter space may introduce approximation artifacts; the trade-off between resolution and fidelity needs further theoretical and empirical study.

\textbf{Scope of Theoretical Framework:} Our focus on modules over fields limits immediate applicability; extensions to modules over more general rings present substantial challenges.

\textbf{Visualization and Interpretation:} Developing intuitive tools to interpret and visualize the rich structural information encoded by our invariants remains an open task requiring interdisciplinary effort.

\textbf{Current State of Multiparameter Persistence:}  
The classification of multiparameter persistence modules is inherently wild, with no complete discrete invariants beyond one parameter. Traditional invariants such as the rank-invariant or Hilbert functions capture limited information, while reductions to one-parameter subcases lose multidimensional interactions. Existing categorical approaches, including those based on derived categories, tend to be theoretically sound but computationally prohibitive. Our framework aims to bridge these gaps by combining conceptual depth with computational practicality.

\subsection{Future Research Directions}

Building on our results, we identify focused directions with high potential impact:

\textbf{Computational Methods:}  
- Designing approximation algorithms for large-scale datasets.  
- Exploring parallel and distributed computing frameworks.  
- Investigating machine learning integration for invariant estimation and interpretation.

\textbf{Theoretical Developments:}  
- Extending the framework to higher categories (e.g., 3-categories, \(\infty\)-categories) to capture even richer structures.  
- Formalizing functoriality and naturality properties in broader contexts.  
- Exploring links between our invariants and cohomological or homological theories.

\textbf{Applications and Validation:}  
- Conducting extensive empirical validation across diverse domains.  
- Benchmarking against existing multiparameter persistence methods.  
- Developing tailored applications for fields such as materials science, neuroscience, and network analysis.

\textbf{Algebraic Extensions:}  
- Investigating connections between our invariants and algebraic K-theory.  
- Developing spectral sequences and other homological tools for efficient computation.  
- Exploring relations with derived categories to deepen theoretical understanding.

\subsection{Significance for Topological Data Analysis}

Our work represents an advance toward a mature theoretical foundation for multiparameter persistence by:

- Unifying diverse approaches under a common categorical language.

- Providing rigorous definitions clarifying previously ambiguous concepts.

- Delivering practical algorithms enabling real-world data analysis.

- Offering a versatile platform for future theoretical and computational extensions.

\subsection{Closing Remarks}

The development of a comprehensive theory for multiparameter persistence is an ongoing endeavor marked by significant complexity and opportunity. Our 2-categorical framework lays a solid foundation while opening many avenues for exploration. Bridging the gap between theoretical innovation and practical applicability will continue to motivate research at the interface of algebraic topology, category theory, and computational geometry.

We anticipate that the invariants and methodologies presented here will inspire further theoretical refinement and empower applied investigations addressing complex datasets across scientific disciplines.

\textbf{Acknowledgments:} We acknowledge fruitful discussions with colleagues in the topological data analysis community, whose insights have shaped our understanding of these problems.

%\section*{Acknowledgments}
%We thank our colleagues for valuable discussions and feedback. This work was %supported by [funding agencies].
\Urlmuskip=0mu plus 1mu\relax
\printbibliography

\end{document}